\documentclass[11pt]{article}
\usepackage[centertags]{amsmath}
\usepackage{amsfonts}
\usepackage{amssymb}
\usepackage{amsthm,color}
\usepackage{newlfont}
\usepackage{bbm}
\usepackage{graphicx}

\newtheorem{Theorem}{Theorem}[section]
\newtheorem{Definition}[Theorem]{Definition}
\newtheorem{Proposition}[Theorem]{Proposition}
\newtheorem{Lemma}[Theorem]{Lemma}
\newtheorem{Corollary}[Theorem]{Corollary}
\newtheorem{Remark}[Theorem]{Remark}

\newtheorem{Hypothesis}{Hypothesis}

\numberwithin{equation}{section}

\begin{document}
\renewcommand{\figurename}{Fig.1}

\def\r2{\mathbb{R}^2}
\def\le{\left}
\def\r{\right}
\def\cost{\mbox{const}}
\def\a{\alpha}
\def\d{\delta}
\def\ph{\varphi}
\def\e{\epsilon}
\def\la{\lambda}
\def\si{\sigma}
\def\La{\Lambda}
\def\B{{\cal B}}
\def\A{{\mathcal A}}
\def\L{{\mathcal L}}
\def\O{{\mathcal O}}
\def\bO{\overline{{\mathcal O}}}
\def\F{{\mathcal F}}
\def\K{{\mathcal K}}
\def\H{{\mathcal H}}
\def\D{{\mathcal D}}
\def\C{{\mathcal C}}
\def\M{{\mathcal M}}
\def\N{{\mathcal N}}
\def\G{{\mathcal G}}
\def\T{{\mathcal T}}
\def\R{{\mathbb R}}
\def\I{{\mathcal I}}

\def\bw{\overline{W}}
\def\phin{\|\varphi\|_{0}}
\def\s0t{\sup_{t \in [0,T]}}
\def\lt{\lim_{t\rightarrow 0}}
\def\iot{\int_{0}^{t}}
\def\ioi{\int_0^{+\infty}}
\def\ds{\displaystyle}
\def\pag{\vfill\eject}
\def\fine{\par\vfill\supereject\end}
\def\acapo{\hfill\break}

\def\beq{\begin{equation}}
\def\eeq{\end{equation}}
\def\barr{\begin{array}}
\def\earr{\end{array}}
\def\vs{\vspace{.1mm}   \\}
\def\rd{\reals\,^{d}}
\def\rn{\reals\,^{n}}
\def\rr{\reals\,^{r}}
\def\bD{\overline{{\mathcal D}}}
\newcommand{\dimo}{\hfill \break {\bf Proof - }}
\newcommand{\nat}{\mathbb N}
\newcommand{\E}{\mathbb E}
\newcommand{\Pro}{\mathbb P}
\newcommand{\com}{{\scriptstyle \circ}}
\newcommand{\reals}{\mathbb R}

\def\Amu{{A_\mu}}
\def\Qmu{{Q_\mu}}
\def\Smu{{S_\mu}}
\def\H{{\mathcal{H}}}
\def\Im{{\textnormal{Im }}}
\def\Tr{{\textnormal{Tr}}}
\def\E{{\mathbb{E}}}
\def\P{{\mathbb{P}}}
\def\span{{\textnormal{span}}}
\title{Fast flow asymptotics for stochastic incompressible viscous fluids in $\mathbb{R}^2$ and SPDEs on  graphs
}
\author{Sandra Cerrai\thanks{Partially supported by the NSF grant DMS 1407615 {\em Asymptotic Problems for SPDEs}.}, Mark Freidlin\thanks{Partially supported by the NSF grant DMS 1411866 {\em Long-term Effects of Small Perturbations and Other Multiscale Asymptotic Problems}.}\\
\vspace{.1cm}\\
Department of Mathematics\\
 University of Maryland\\
College Park, Maryland, USA
}

\date{}

\maketitle

\begin{abstract}
Fast advection asymptotics for a stochastic reaction-diffusion-advection equation are studied in this paper. To describe the asymptotics, one should consider a suitable class of SPDEs defined on a graph, corresponding to the stream function of the underlying incompressible flow.

\end{abstract}

\section{Introduction}
\label{sec1}

Consider an incompressible flow in $\mathbb{R}^2$, with stream function $-H(x)$, $x \in\,\mathbb{R}^2$, and let some particles move together with the flow. If we denote by $u(t,x)$ the density of the particles at time $t\geq 0$ and position $x \in\,\mathbb{R}^2$, then the function $u(t,x)$ satisfies the Liouville equation
\begin{equation}
\label{adv}
\le\{\begin{array}{l}
\ds{\partial_t u(t,x)=\le<\bar{\nabla}H(x),\nabla u(t,x)\r>,\ \ \ \ \ t>0,\ \ \ x \in\,\mathbb{R}^2,}\\
\vs
\ds{u(0,x)=\varphi(x),\ \ \ \ \ x \in\,\mathbb{R}^2.}
\end{array}\r.\end{equation}
Suppose now that the flow has a small viscosity and the particles take part in a slow chemical reaction, with a deterministic and a stochastic component, as described by the equation
\begin{equation}
\label{eq1-intro-pre}
\le\{\begin{array}{l}
\ds{\partial_t \tilde{u}_\e(t,x)=\frac \e2\,\Delta \tilde{u}_\e(t,x)+\le<\bar{\nabla}H(x),\nabla \tilde{u}_\e(t,x)\r>+\e\,b(\tilde{u}_\e(t,x))+\sqrt{\e}\,g(\tilde{u}_\e(t,x))\partial_t \tilde{\mathcal{W}}(t,x),}\\
\vs
\ds{\tilde{u}_\e(0,x)=\varphi(x),\ \ \ x \in\,\mathbb{R}^2.}
\end{array}\r.
\end{equation}
Here, $0<\e<<1$ is a small parameter, $b, g:\reals^2\to \reals$ are Lipschitz continuous non-linearities and the stream function $-H:\reals^2\to\reals$ is a generic function, having four continuous derivatives, with bounded second derivative, and such that $H(x)\to +\infty$, as $|x|\uparrow +\infty$. The noise $\tilde{\mathcal{W}}(t,x)$ is supposed to be a spatially homogeneous Wiener process having finite spectral measure (see Sections \ref{sec2} and \ref{sec5} for all assumptions and details).

The small positive parameter $\e$ is included in equation \eqref{eq1-intro-pre} in such a  way that all perturbation terms have strength of the same order, as $\e\downarrow 0$.  It is not difficult to check that under the above conditions, if we take the limit as $\e\downarrow 0$,  the solution  $\tilde{u}_\e(t,x)$ of equation \eqref{eq1-intro-pre} converges on any finite time interval to the solution $u(t,x)$ of equation \eqref{adv},  in probability, uniformly with respect to $x$ in a bounded domain of $\mathbb{R}^2$. But on large time intervals, growing together with $\e^{-1}$, the difference $\tilde{u}_\e(t,x)-u(t,x)$ can have order $1$, as $\e\downarrow 0$.

To describe the long-time behavior of the particle density, we define
\[u_\e(t,x)=:\tilde{u}_\e(t/\e,x),\ \ \ \ \ t\geq 0,\ \ \ x \in\,\mathbb{R}^2.\]
With this change of time, the new function $u_\e(t,x)$ solves the equation
\begin{equation}
\label{eq1-intro}
\le\{\begin{array}{l}
\ds{\partial_t u_\e(t,x)=\frac 12\,\Delta u_\e(t,x)+\frac 1\e\le<\bar{\nabla}H(x),\nabla u_\e(t,x)\r>+b(u_\e(t,x))+g(u_\e(t,x))\partial_t \mathcal{W}(t,x),}\\
\vs
\ds{u_\e(0,x)=\varphi(x),\ \ \ x \in\,\mathbb{R}^2,}
\end{array}\r.
\end{equation}
for some spatially homogeneous Wiener process $\mathcal{W}(t,x)$.

In the present paper, we are interested in the limiting behavior  of the solution $u_\e(t,x)$ of equation \eqref{eq1-intro}, as $\e\downarrow 0$, in a finite time interval. In particular, we will see that in order to describe the limit of $u_\e(t,x)$, one should consider SPDEs on a non standard setting, where the space variable changes on the  graph $\Gamma$ obtained by identifying all points in each connected component of the level sets of the Hamiltonian $H$.

A suitable class of SPDEs on a graph has been already studied in our previous paper \cite{cf16}, where SPDEs defined on a net of narrow channels were studied. In that case, we have tried to understand what happens of the solution of the SPDE defined on a $2$-dimensional  channel $G$ with many wings and subject to instantaneous reflections at the boundary, when the width of the channel goes to zero. Actually, we have proved that the solution converges to the solution of a suitable SPDE, defined on a suitable graph that can be associated with the channel, in $L^p(\Omega; C([\tau,T];L^2(G)))$, for any $0<\tau<T$.

Here we are considering the case of a reaction-diffusion-advection equation in $\mathbb{R}^2$, where the reaction term has a deterministic and a stochastic component, and the advection term is  of order $\e^{-1}$, compared to the diffusion and the reaction part. 
For every fixed $\e>0$, the second order differential operator $\mathcal{L}_\e$ defined by
\[\mathcal{L}_\e \varphi(x)=\frac 12 \Delta \varphi(x)+\frac 1\e\le<\bar{\nabla}H(x),\nabla \varphi(x)\r>,\ \ \ \ x \in\,\mathbb{R}^2,\]
is associated with the  stochastic equation
\begin{equation}
\label{stoch-e}
d{X}_\e(t)=\frac 1\e\,\bar{\nabla}H({X}_\e(t))\,dt+d{w}(t),\ \ \ \ {X}_\e(0)=x \in\,\mathbb{R}^2,\end{equation}
for some $2$-dimensional  Brownian motion $w(t)$, defined on a stochastic basis $(\Omega,\mathcal{F}, \{\mathcal{F}_t\}_{t\geq 0}, \mathbb{P})$.
This means that  $u_\e$ is a mild solution to equation \eqref{eq1-intro}  if it satisfies
\[u_\e(t)=S_\e(t)\varphi+\int_0^t S_\e(t-s) B(u_\e(s))\,ds+\int_0^t S_\e(t-s)\,G(u_\e(s))\,d\mathcal{W}(s),\]
where $B$ and $G$ are the composition/multiplication functionals associated with $b$ and $g$, respectively, (see Section \ref{sec5} for the definition), and $S_\e(t)$ is the Markov transition semigroup associated with equation \eqref{stoch-e}
\[S_\e(t)\varphi(x)=\mathbb{E}_x \varphi(X_\e(t)),\ \ \ \ t\geq 0,\ \ \ x \in\,\mathbb{R}^2.\]

In \cite[Chapter 8]{fw12} it is proved that, if  $\Pi$  is the projection of $\mathbb{R}^2$ onto the graph $\Gamma$, then
for any $x \in\,\mathbb{R}^2$ and $T>0$ the process $\Pi(X_\e(\cdot))$ converges, in the sense of weak convergence of distributions in $C([0,T];\Gamma)$, to a Markov process $\bar{Y}$ on $\Gamma$. Namely, for every continuous functional $F$ defined on $C([0,T];\Gamma)$ and any $x \in\,\mathbb{R}^2$ it holds
\begin{equation}
\label{limwek}
\lim_{\e\to 0}
\mathbb{E}_x\,F(\Pi(X_\e(\cdot)))=\bar{\mathbb{E}}_{\Pi(x)}\,F(\bar{Y}(\cdot)).\end{equation}
The generator $\bar{L}$ of the process $\bar{Y}$ is explicitly given, in terms of suitable second order differential operators defined on each edge of the graph and suitable gluing conditions at the vertices.

As a consequence of the limiting result  \eqref{limwek}, in \cite[Chapter 8]{fw12} Freidlin and Wentcell have also studied the asymptotic behavior of the solution of the elliptic problem
\[\le\{\begin{array}{l}
\ds{\frac 12\Delta f_\e(x)+\frac 1\e \le<\bar{\nabla}H(x),\nabla f_\e(x)\r>=-g,\ \ \ \ x \in\,D,}\\
\vs
\ds{f_\e(x)=\rho(x),\ \ \ \ x \in\,\partial D,}
\end{array}\r.\]
where $D$ is a bounded smooth domain in $\mathbb{R}^2$ and $g$ and $\rho$ are continuous functions on $D$ and $\partial D$, respectively. Actually, they have proven that  $f_\e$ converges to  the solution of the corresponding elliptic equation on the graph, associated with the operator $\bar{L}$. In \cite{IS},  Ishii and Souganidis, by using only deterministic arguments, have proved an analogous result in  the more general situation the Laplace operator is replaced with the operator
$\mbox{Tr}[A(x)D^2]$, where $A$ is a smooth, symmetric, non-negative matrix-valued mapping defined on $D$  

Next,  in \cite{f02} the limiting behavior of the solution of the deterministic parabolic problem 
\begin{equation}
\label{det-intro}
\le\{\begin{array}{l}
\ds{\partial_t v_\e(t,x)=\frac 12\,\Delta v_\e(t,x)+\frac 1\e\le<\bar{\nabla}H(x),\nabla v_\e(t,x)\r>+b(v_\e(t,x)),}\\
\vs
\ds{v_\e(0,x)=\varphi(x),\ \ \ x \in\,\mathbb{R}^2,}
\end{array}\r.
\end{equation}
has been studied.
Under the crucial assumption that  the projection  of the support of the function $\varphi$ on the graph $\Gamma$ does not contain any vertex, it is shown that  for any $0<\tau<T$
\begin{equation}
\label{fund-intro}
\lim_{\e\to 0} \sup_{t \in\,[\tau,T]}|u_\e(t,x)-\bar{v}(t,\Pi(x))|=0,\end{equation}
uniformly with respect to  $x$ from any compact set of $\mathbb{R}^2$, where $\bar{v}$ is the solution of the parabolic problem on $\Gamma$
\begin{equation}
\label{graph-intro}
\le\{\begin{array}{l}
\ds{\partial_t \bar{v}(t,z,k)=\bar{L} \bar{v}(t,z,k)+b(\bar{v}(t,z,k)),}\\
\vs
\ds{\bar{v}(0,z,k)=\varphi^\wedge(z,k): =\frac 1{T_k(z)}\oint_{C_k(z)}\frac{\varphi(x)}{|\nabla H(x)|}\,dl_{z,k}. }
\end{array}\r.
\end{equation}
Here $dl_{z,k}$ is the surface measure on the connected component $C_k(z)$ of the level set $C(z)=\{ x \in\,\mathbb{R}^2\,:\,H(x)=z \}$, corresponding to the edge $I_k$, and
\[T_k(z)=\oint_{C_k(z)}\frac{1}{|\nabla H(x)|}\,dl_{z,k},\]
(see Section \ref{sec2} for all details).

Assuming that the projection  of the support of the initial condition  $\varphi$ on the graph $\Gamma$ does not contain any vertex, allows to avoid to deal with the vertices points, where serious discontinuity problems arise. Actually, in order to prove \eqref{fund-intro}, it is necessary to prove that for any $\varphi \in\,C_b(\mathbb{R}^2)$
\begin{equation}
\label{intro1}
\lim_{\e\to 0} \sup_{t \in\,[\tau,T]}\le|\mathbb{E}_x \le[\varphi(X_\e(t))-\varphi^\wedge(\Pi(X_\e(t)))\r]\r|=0,
\end{equation}
and
\begin{equation}
\label{intro2}
\lim_{\e\to 0}\,\sup_{t \in\,[\tau,T]}\le|\mathbb{E}_x\,\varphi^\wedge(\Pi(X_\e(t)))-\bar{\mathbb{E}}_{\Pi(x)}\,\varphi^\wedge (\bar{Y}(t))\r|=0.
\end{equation}
Limit \eqref{intro1} has been obtained in \cite{f02}, as a consequence of the averaging principle, by using angle-action coordinates, away from the vertices. Limit \eqref{intro2} was obtained in \cite{f02} as an immediate consequence of  \eqref{limwek},  since the function $\varphi^\wedge$ is continuous away from the vertices.

But the assumption that the support of the initial condition  on the graph $\Gamma$ does not contain any vertex  it too restrictive and it is critical in several situations to be able to prove \eqref{fund-intro} for a general $\varphi \in\,C_b(\mathbb{R}^2)$. In particular, this is necessary when dealing with SPDEs, as in this case we cannot assume that the support of the noise satisfies such a condition.

For this reason, in Section \ref{sec4} we prove that for any general function $\varphi \in\,C_b(\mathbb{R}^2)$
\begin{equation}
\label{fund-intro}
\lim_{\e\to 0} \sup_{t \in\,[\tau,T]}\le|S_\e(t)u(x)-(\bar{S}(t) u^\wedge)\circ \Pi (x)\r|=0,\end{equation}
with $x \in\,\mathbb{R}^2$ and $0<\tau<T$ fixed. Also in this case, \eqref{fund-intro} follows once we prove limits \eqref{intro1} and \eqref{intro2}, but their proofs are considerably  more delicate than in \cite{f02}, due to the presence of vertices. Actually, in order to prove \eqref{intro1} and \eqref{intro2} we have to consider separately  the region of $\mathbb{R}^2$ close to the critical points of the Hamiltonian and the region far from them and introduce suitable sequences of stopping times that allow to go from one region to the other.  By using the fact that the process $X_\e$ spends a small amount of time close to the critical points, we  obtain suitable nice properties of those stopping times that allow us to conclude the validity of \eqref{fund-intro} for general functions $\varphi$.

Next,  we go back to the SPDE \eqref{eq1-intro}, where, as  we mentioned above,  $\mathcal{W}(t,x)$ is a space homogeneous Wiener process, having finite spectral measure $\mu$. This means, that we can represent $\mathcal{W}(t,x)$ as 
\[\mathcal{W}(t,x)=\sum_{j=1}^\infty\, \widehat{u_j \mu}(x)\,\beta_j(t),\ \ \ \ t\geq 0,\]
for some  complete orthonormal system $\{u_j\}_{j \in\,\nat}$ in $L^2(\mathbb{R}^2;d\mu)$ and a sequence of independent Brownian motions $\{\beta_j\}_{j \in\,\nat}$, all defined on the same stochastic basis.

We study equation \eqref{eq1-intro} in a space of square integrable functions, with respect to a weighted measure $\gamma\circ \Pi(x)\,dx$ on $\mathbb{R}^2$. In fact, the choice of the weight $\gamma$  is not trivial, as we have   to choose it in such a way that $\gamma\circ \Pi$ is admissible with respect to all semigroups $S_\e(t)$, its projection  $\gamma$ on the graph $\Gamma$  is admissible with respect to the semigroup $\bar{S}(t)$ and functions on $L^2(\mathbb{R}^2,\gamma\circ \Pi\,dx)$ are projected to functions in $L^2(\gamma,d\nu_\gamma)$, where $\nu_\gamma$ is the projection  of the Lebesgue measure on $\Gamma$, with weight $\gamma$. Moreover, we need to show that from \eqref{fund-intro} we obtain the limit
\begin{equation}
\label{ultimo}
\lim_{\e\to 0} \sup_{t \in\,[\tau,T]}\le|S_\e(t)u-(\bar{S}(t) u^\wedge)\circ \Pi \r|_{L^2(\mathbb{R}^2,\gamma\circ \Pi(x)dx)}=0.
\end{equation}

Once identified the right class of weights, we introduce the process 
\[\bar{\mathcal{W}}(t,z,k)=\sum_{j=1}^\infty\, (\widehat{u_j \mu})^\wedge (z,k)\,\beta_j(t),\ \ \ \ t\geq 0\ \ \ (z,k) \in\,\Gamma,\]
and we show that $\bar{\mathcal{W}}(t) \in\,L^2(\boldsymbol{\Omega};L^2(\Gamma,d\nu_\gamma))$, for every $t\geq 0$. In particular, as we can prove that
\[\|\bar{S}(t)\|_{\mathcal{L}(L^2(\Gamma,d\nu_\gamma))}\leq c_T,\ \ \ \ t \in\,[0,T],\]
and $g$ is Lipschitz continuous, we obtain that for any $v \in\,L^p(\boldsymbol{\Omega};C([0,T];L^2(\Gamma,d\nu_\gamma)))$ the stochastic convolution
\[ t \in\,[0,+\infty)\mapsto \int_0^t \bar{S}(t-s)\,G(v(s))\,d\bar{\mathcal{W}}(s)=
\sum_{k=1}^\infty\int_0^t \bar{S}(t-s) \le[G(v(s))(\widehat{u_j \mu})^\wedge\r]\,d\beta_j(s),\] is well defined in $L^2(\boldsymbol{\Omega};C([0,T];L^2(\Gamma,d\nu_\gamma)))$ and the mapping 
\[u \mapsto \int_0^{\cdot}\bar{S}(\cdot-s)\,G(v(s))\,d\bar{\mathcal{W}}(s)\]
is Lipschitz continuous in $L^2(\boldsymbol{\Omega};C([0,T];L^2(\Gamma,d\nu_\gamma)))$. 

In particular, as also $b$ is Lipschitz continuous, this implies that the SPDE 
\[ \le\{\begin{array}{l}
\ds{\partial_t \bar{u}(t,z,k)=\bar{{L}}\,\bar{u}(t,z,k)+b(\bar{u}(t,z,k))+g(\bar{u}(t,z,k))\,\partial_t \bar{\mathcal{W}}(t,z,k),\ \ \ \ t>0,}\\
\vs
\ds{\bar{u}(0,z,k)=\varphi^\wedge(z,k),\ \ \ (z,k) \in\,\Gamma.}
\end{array}\r.\]
is well posed in $L^2(\boldsymbol{\Omega};C([0,T];L^2(\Gamma,d\nu_\gamma)))$. Finally, as a consequence of \eqref{ultimo}, we can prove the main result of this paper, namely
\[\lim_{\e\to 0}\mathbf{E}\sup_{ t \in\,[\tau,T]}|u_\e(t)-\bar{u}(t)\circ \Pi\,|^p_{L^2(\mathbb{R}^2,\gamma\circ \Pi(x)dx)}=0.\]

\bigskip

To conclude, we would like to stress the fact that the techniques we are currently developing to deal with SPDEs on graphs can be further developed to treat more sophisticated and complex situations. For example, one could consider the case there are several different types of particles, instead of only one,  as in the present paper. In this case we expect to get a system of SPDEs on a graph. Moreover, if the original flow does not moves on $\mathbb{R}^2$, but on a $2$-dimensional surface (for example a $2D$  torus), with a positive genus, the underling dynamics  can have delays at some vertices. This leads to certain effects for SPDEs on the graph that are worth of investigation. Finally, in the multidimensional case, when several conservation laws are present, we expect will obtain SPDEs on a generalization of a graph, namely an open book (see \cite[Chapter 9]{fw12}). We are planning to address in forthcoming papers.

\section{Notations and preliminaries}
\label{sec2}

We consider here the Hamiltonian system 
\begin{equation}
\label{ham}
\dot{X}(t)=\bar{\nabla} H(X(t)),\ \ \ \ \ X(0)=x \in\,\mathbb{R}^2.
\end{equation}

Throughout the present paper, we will assume that the Hamiltonian $H$ is a generic function, with non degenerate critical points, and having quadratic growth, for $|x|\to \infty$. More precisely
\begin{Hypothesis}
\label{H1}
The mapping $H:\mathbb{R}^2\to\mathbb{R}$ satisfies the following conditions.
\begin{enumerate}
\item It is four times continuously differentiable, with bounded second derivative.
\item It has only a finite number of critical points $x_1,\ldots, x_n$. The matrix of second derivatives $D^2H(x_i)$ is non degenerate, for every $i=1,\ldots,n$ and $H(x_i)\neq H(x_j)$,  if $i\neq j$.
\item There exist three positive constants $a_1, a_2, a_3$ such that
$H(x)\geq a_1\,|x|^2$, $|\nabla H(x)|\geq a_2\,|x|$ and $\Delta H(x)\geq a_3$, for all $x \in\,\reals^2$, with $|x|$ large enough.
\item We have 
\[\min_{x \in\,\mathbb{R}^2}H(x)=0.\]

\end{enumerate}

\end{Hypothesis}

Notice that we can always assume condition 4. without any loss of generality.

\medskip

Once introduced the Hamiltonian $H$, for every $z\geq 0$, we denote by $C(z)$ the $z$-level set 
\[C(z)=\le\{x \in\,\mathbb{R}^2\,:\,H(x)=z\r\}.\]
The set  $C(z)$ may consist of several connected components
\[C(z)=\bigcup_{k=1}^{N(z)} C_k(z),\]
and for every $x \in\,\mathbb{R}^2$  we have
\[X(0)=x\Longrightarrow X(t) \in\,C_{k(x)}(H(x)),\ \ \ \ t\geq 0,\]
where $C_{k(x)}(x)$ is the  connected component of the level set $C(H(x))$, to which  the point $x$ belongs.
For every $z \geq 0$ and $k=1,\ldots,N(z)$, we shall denote by $G_k(z)$ the domain of $\mathbb{R}^2$ bounded by the level set component $C_k(z)$.

\subsection{The graph $\Gamma$ and the identification map $\Pi$}

If we identify all points in $\mathbb{R}^2$  belonging to the same connected component of a given level set $C(z)$ of the Hamiltonian $H$, we obtain a graph $\Gamma$,  given by several intervals $I_1,\ldots I_n$ and vertices $O_1,\ldots, O_m$. The vertices will be of two different types,  external and internal vertices. External vertices correspond to local extrema of  $H$, while internal vertices correspond to saddle points of $H$. Among external vertices, we will also include $O_\infty$, the endpoint of the only unbounded interval in the graph, corresponding to the point at infinity.

In what follows, we shall denote by $\Pi:\mathbb{R}^2\to \Gamma$ the {\em identification map}, that associates to every point $x \in\,\mathbb{R}^2$ the corresponding point $\Pi(x)$ on the graph $\Gamma$. We have $\Pi(x)=(H(x),k(x))$, where $k(x)$ denotes the number of the interval on the graph $\Gamma$, containing the point $\Pi(x)$. If $O_i$ is one of the interior vertices, the second coordinate cannot  be chosen in a unique way, as there are three edges having $O_i$ as their endpoint. Notice that both $k(x)$ and $H(x)$ are  first integrals (a discrete and a continuous one, respectively) for  the Hamiltonian system \eqref{ham}.

On the graph $\Gamma$, a distance can be introduced in the following way. If $y_1=(z_1,k)$ and $y_2=(z_2,k)$ belong to the same edge $I_k$, then $d(y_1,y_2)=|z_1-z_2|$. In the case $y_1$ and $y_2$ belong to different edges, then
\[d(y_1,y_2)=\min\,\le\{d(y_1, O_{i_1})+d(O_{i_1},O_{i_2})+\cdots+d(O_{i_j},y_2)\r\},\]
where the minimum is taken over all possible paths from $y_1$ to $y_2$, through every possible sequence of vertices $O_{i_1},\ldots,O_{i_{j}}$, connecting $y_1$ to $y_2$.

\subsection{Some other notations and preliminary facts related to the Hamiltonian}
\label{ss2.2}

If $z$ is not a critical value, then each $C_k(z)$ consists of one periodic trajectory of the vector field $\bar{\nabla}H(x)$. If $z$ is a local extremum of $H(x)$, then, among the components of $C(z)$ there is a set consisting of one point, the rest point of the flow. If $H(x)$ has a saddle point at some point $x_0$ and $H(x_0)=z$, then $C(z)$ consists of three trajectories, the equilibrium point $x_0$ and the two trajectories that have $x_0$ as their limiting point, as $t\to \pm \infty$.

\medskip

We introduce here some other notations related to the Hamiltonian $H$, that will be used throughout the paper.
\begin{itemize}
\item[-] For every edge $I_k$, we denote 
\[G_k =\le\{ x \in\,\mathbb{R}^2\,:\ \Pi(x) \in\,\mathring{I}_k \r\}.\]
\item[-] For every $0\leq z_1<z_2$ and for every edge $I_k$, we denote
\[G_k(z_1,z_2)=G_k(z_2,z_1)=\le\{x \in\,G_k\,:\,z_1<H(x)<z_2\r\}.\]
\item[-] For every $\d>0$, we denote
\[G(\pm \d)=\bigcup_{i=1}^m G^{\,i}(\pm \d)=\bigcup_{i=1}^m\le\{ x \in\,\mathbb{R}^2\,:\,H(O_i)-\d<H(x)<H(O_i)+\d\r\},\]
(here, with some abuse of notation, we denote by $H(O_i)$ the value of the Hamiltonian $H$ at those $x \in\,\mathbb{R}^2$ such that $\Pi(x)=O_i$).
\item[-] For every vertex $O_i$ and edge $I_k$, we  denote
\[D^{\,i}=\le\{ x \in\,\mathbb{R}^2\,:\ \Pi(x) =O_i\r\},\ \ \ \ \ D_{k}^{\,i}=D^{\,i}\cap \bar{G}_k.\]
\item[-] For every value $z\geq 0$ taken by the function $H$ on the set $\bar{G}_k$, we denote
\[C_k(z)=\le\{ x \in\, \bar{G}_k\,:\,H(x)=z\r\}=C(z)\cap \bar{G}_k.\]
\item[-] For every $\d>0$ and every edge $I_k$ and vertex $O_i$ such that $I_k\sim O_i$, we denote
\[D(\pm \d)=\bigcup_{k, i \,:\, I_k\sim O_i} D_{k}^{\,i}(\pm \d)=\bigcup_{k, i\, :\, I_k\sim O_i}\,\le\{ x \in\,G_k\,:\ \text{dist}(\Pi(x) ,O_i)=\d\r\}.\]

\end{itemize}

Now, for every $(z,k) \in\,\Gamma$, we define
\begin{equation}
\label{fluid2}
T_k(z)=\oint_{C_k(z)}\frac 1{|\nabla H(x)|}\,dl_{z,k},\ \ \ \ \a_k(z)=\oint_{C_k(z)}|\nabla H(x)|\,dl_{z,k},
\end{equation}
where $dl_{z,k}$ is the length element on $C_k(z)$. Notice that $T_k(z)$ is the period of the motion along the level set $C_k(z)$.

As we have seen above, if $X(0)=x \in\,C_k(z)$, then $X(t) \in\,C_k(z)$, for every $t\geq 0$. As  known, for every $(z,k) \in\,\Gamma$ the probability measure
\begin{equation}
\label{brr}
d\mu_{z,k}:=\frac 1{T_k(z)}\,\frac 1{|\nabla H(x)|}\,dl_{z,k}\end{equation}
is invariant for system \eqref{ham} on the level set $C_k(z)$. Moreover, it is possible to prove that for any $0\leq z_1<z_2$ and $k=1,\ldots n$
\begin{equation}
\label{fluid10}
\int_{G_k(z_1,z_2)}u(x)\,dx=\int_{I_k\cap [(z_1,k),(z_2,k)]}\oint_{C_k(z)}\frac{u(x)}{|\nabla H(x)|}\,dl_{z,k}\,dz.
\end{equation}
In particular, if we take $u\equiv 1$, we get  
\[\text{area}(G_k(z_1,z_2))=\int_{I_k\cap [(z_1,k),(z_2,k)]} T_k(z)\,dz,\]
so that
\begin{equation}
\label{fluid3}
T_k(z)=\frac{dS_k(z)}{dz},\end{equation}
where $S_k(z)$ is the area  of the domain $G_k(z)$ bounded by the level set $C_k(z)$.
Finally, by the divergence theorem, it is immediate to check that
\[\a_k(z)=\le|\int_{G_k(z)}\Delta H(x)\,dx\r|=\le|\int_{G_k(z)}\omega_H(x)\,dx\r|,\]
where $\omega_H$ is the vorticity of the flow.

As discussed in \cite[section 8.1]{fw12}, since the Hamiltonian $H$ has only non-degenerate critical points, if $(z,k)$ approaches the endpoint of an edge $I_k$, corresponding to an external vertex $O_i=(H(x_j),k)$, then
\begin{equation}
\label{fluid4}
\lim_{(z,k)\to O_i} T_k(z)=\frac{2\,\pi}{\sqrt{\text{det}\, [\text{Hess} H(x_j)} ]}>0.
\end{equation}
If $(z,k)$ approaches the endpoint of an edge $I_k$, corresponding to an internal vertex $O_i=(H(x_j),k)$, then
\begin{equation}
\label{fluid5}
T_k(z)=\oint_{C_k(z)}\frac 1{|\nabla H(x)|}\,dl_{z,k}\sim \text{const}\,\le|\log\le|z-H(x_i)\r|\r|\to +\infty,\ \ \ \text{as}\ (z,k)\to O_i.\end{equation}
Finally, if $(z,k)$ approaches $O_\infty$, we have
\begin{equation}
\label{fluid6}
T_k(z)=O(1),\ \ \ \ \text{as}\ (z,k)\to O_\infty.\end{equation}
Actually,  for $|x|$ large, we have $H(x)\geq a_1\,|x|^2$, so that $S_k(z)=O(z)$ and hence, due to \eqref{fluid3},  \eqref{fluid6} follows.  

Next, we would like to recall that in \cite[Lemma 8.1.1]{fw12} it is proven that if $v \in\,C^1(\mathbb{R}^2)$, then for every fixed $k$ the mapping 
\[z \mapsto \oint_{C_k(z)}v(x)\,|\nabla H(x)|\,dl_{z,k},\]
is continuously differentiable with respect to $z$ such that $(z,k) \in\,\mathring{I}_k$, and 
\begin{equation}
\label{f50}
\frac d{dz}\oint_{C_k(z)}v(x)\,|\nabla H(x)|\,dl_{z,k}=\oint_{C_k(z)}\frac {\langle \nabla v(x),\nabla H(x)\rangle+v(x)\Delta H(x)}{|\nabla H(x)|}\,dl_{z,k}.\end{equation}

\begin{figure}
\centering
\includegraphics[height=8cm, width=12cm, bb=38 6 459 357]{Figure_1.eps}
\caption{The Hamiltonian, the level sets, the projection and the graph}
\end{figure}

\subsection{A limiting result}
\label{ssec2.3}

We consider here  the following random perturbation of system \eqref{ham}
\begin{equation}
\label{ham-epsilon}
d\tilde{X}_\e(t)=\bar{\nabla}H(\tilde{X}_\e(t))\,dt+\sqrt{\e}\,d\tilde{w}(t),\ \ \ \ \tilde{X}_\e(0)=x,
\end{equation}
where $\tilde{w}(t)$ is a two dimensional Wiener process and $\e>0$ is a small parameter. Because of the perturbation, the mapping $t\in\,[0,+\infty)\mapsto H(\tilde{X}_\e(t)) \in\,\mathbb{R}$ is not constant any more. Actually, the motion $\tilde{X}_\e(t)$ consists of a fast rotation along the deterministic unperturbed  trajectories and a slow motion across them. 

In what follows, it will be convenient to do a change of time and, with the time rescaling $t\mapsto t/\e$,  the process $\tilde{X}_\e(t/\e)$ will coincide in distribution with the solution of the stochastic equation
\begin{equation}
\label{ham-epsilon-bis}
d{X}_\e(t)=\frac 1\e\,\bar{\nabla}H({X}_\e(t))\,dt+d{w}(t),\ \ \ \ {X}_\e(0)=x,
\end{equation}
where ${w}(t)$ is another two dimensional Wiener process.  We will denote by $S_\e(t)$, $t\geq 0$, the Markov transition semigroup associated with equation \eqref{ham-epsilon-bis}. Namely, for every Borel bounded function $\varphi:\mathbb{R}^2\to \mathbb{R}$
\[S_\e(t)\varphi(x)=\mathbb{E}_x\varphi(X_\e(t)),\ \ \ \ t \geq 0,\ \ \ x \in\,\mathbb{R}^2.\]

Now, for every $x \in\,\mathbb{R}^2$, we consider the process $\Pi(X_\e(t))$, $t\geq 0 $, defined on the graph $\Gamma$, with $X_\e(0)=x$. In \cite[Chapter 8]{fw12} it  has been studied the limiting behavior, as $\e\downarrow 0$, of the process $\Pi(X_\e)$ in the space  $C([0,T];\Gamma)$, for any fixed $T>0$ and $x \in\,\mathbb{R}^2$. Namely, in \cite[Theorem 8.2.2]{fw12} it has been proved  that the process $\Pi(X_\e)$, which describes the slow motion of the motion $X_\e$, converges, in the sense of weak convergence of distributions in the space of continuous $\Gamma$-valued functions, to a diffusion process $\bar{Y}$ on $\Gamma$.

The process $\bar{Y}$ has been described in \cite[Theorem 8.2.1]{fw12}  in terms of its generator $\bar{L}$, which is given by suitable differential operators $\bar{{\cal L}}_k$ within each edge $I_k$ of the graph and by certain gluing conditions at the interior vertices $O_i$ of the graph.

For each $k=1,\ldots,n$, the differential operator $\bar{{\cal L}}_k$,  acting on functions $f$ defined on the edge $I_k$,  has the form
\begin{equation}
\label{lk}
\bar{{\cal L}}_k f(z)=\frac 1{2\,T_k(z)}\frac d{dz}\le(\a_k\,\frac{df}{dz}\r)(z),\ \ \ \ 
\end{equation}
where $T_k(z)$ and $\a_k(z)$ are the mappings defined in \eqref{fluid2}. The operator $\bar{L}$, acting on functions $f$ defined on the graph $\Gamma$, is defined as
\[\bar{L}f(z,k)=\bar{{\cal L}}_k f(z),\ \ \ \text{if}\ (z,k)\ \text{is an interior point of the edge $I_k$}.\]
It is immediate to check that each operator $\bar{{\cal L}}_k $ can be represented as $\bar{{\cal L}}_k =(d/dv_k)(d/du_k)$, where
\[v^\prime_k(z)=\oint_{C_k(z)}\frac 1{|\nabla H(x)|}\,dl_{z,k},\ \ \ \ u^\prime_k(z)=\le(\oint_{C_k(z)}|\nabla H(x)|\,dl_{z,k}\r)^{-1}.\]
In view of this representation, by studying the limiting behavior of the function $u_k$ and $v_k$ at the vertices of the graph $\Gamma$, it is possible to check that the internal vertices, corresponding to the saddle points of $H$, are accessible, while the external vertices and the vertex $O_\infty$ are inaccessible.

The domain $D(\bar{L})$ is defined as the set of continuous functions on the graph $\Gamma$, that are twice continuously differentiable in the interior part of each edge of the graph, such that for any vertex $O_i=(z_i,k_{i_1})=(z_i, k_{i_2})=(z_i,k_{i_3})$ there exist finite
\[ \lim_{(z,k_{i_j})\to O_i}\bar{L} f(z,k_{i_j}),\ \ \ \ j=1,2,3,\]
and the limits do not depend on the edge $I_{k_{i_j}}\sim O_i$. Moreover, for each interior vertex $O_i$ the following gluing condition is satisfied
\begin{equation}
\label{gluing}
\sum_{j=1}^{3} \pm\a_{k_{i_j}}(z_i)d_{k_{i_j}}f(z_i,k_{i_j})=0,\end{equation}
where $d_{k_{i_j}}$ is the differentiation along $I_{k_{i_j}}$ and the sign $+$ is taken if the $H$-coordinate increases along $I_{k_{i_j}}$ and the sign $-$ is taken otherwise. 

The operator $(\bar{L}, D(\bar{L})$ is a non-standard operator, because it is a differential operator on a graph, endowed with suitable gluing conditions and because it is degenerate at the vertices of the graph. Nevertheless in  \cite[Theore 8.2.1]{fw12} it is shown that it is the generator of a Markov process $\bar{Y}$ on the graph $\Gamma$. In what follows, we shall denote by $\bar{\mathbb{P}}_{(z,k)}$ and $\bar{\mathbb{E}}_{(z,k)}$ the probability and the expectation associated with $\bar{Y}$, starting from $(z,k) \in\,\Gamma$. Moreover, we shall denote by $\bar{S}(t)$ the semigroup associated with $\bar{Y}$, defined by
\[\bar{S}(t) f(z,k)=\bar{\mathbb{E}}_{(z,k)}f(\bar{Y}(t)),\]
for every bounded Borel function $f:\Gamma\to \mathbb{R}$.

As we mentioned above,  in \cite[Theorem 8.2.2]{fw12} it has been proved that the process $\Pi(X_\e(\cdot))$ is weakly convergent to $\bar{Y}$ in $C([0,T];\Gamma)$, for every $T>0$   and $x \in\,\mathbb{R}^2$. Namely, for every continuous functional $F$ defined on $C([0,T];\Gamma)$ and $x \in\,\mathbb{R}^2$
\begin{equation}
\label{limite-fund}
\lim_{\e\to 0}
\mathbb{E}_x\,F(\Pi(X_\e(\cdot)))=\bar{\mathbb{E}}_{\Pi(x)}\,F(\bar{Y}(\cdot)).
\end{equation}

\section{Functions and operators defined on the graph $\Gamma$}
\label{sec3}

We fix here a continuous mapping $\gamma:\Gamma\to (0,+\infty)$ such that
\begin{equation}
\label{fluid1}
\sum_{k=1}^n \int_{I_k} \gamma(z,k)\,T_k(z)\,dz<\infty,
\end{equation}
where,  we  recall,
\[T_k(z)=\oint_{C_k(z)}\frac 1{|\nabla H(x)|}\,dl_{z,k}.\]
If we define
\[\gamma^{\vee}(x)=\gamma(\Pi(x)),\ \ \ \ x \in\,\mathbb{R}^2,\]
we have that $\gamma^{\vee}:\mathbb{R}^2\to (0,+\infty)$ is a  bounded continuous function in $L^1(\mathbb{R}^2)$.
Actually, due to \eqref{fluid10} we have
\[\begin{array}{l}
\ds{\int_{\mathbb{R}^2} \gamma^{\vee}(x)\,dx =\sum_{k=1}^n \int_{I_k} \oint_{C_k(z}\frac{\gamma^{\vee}(x)}{|\nabla H(x)|}\,dl_{z,k}\,dz=\sum_{k=1}^n \int_{I_k} \gamma(z,k)\oint_{C_k(z}\frac{1}{|\nabla H(x)|}\,dl_{z,k}\,dz}\\
\vs
\ds{=\sum_{k=1}^n \int_{I_k} \gamma(z,k) T_k(z)\,dz<\infty.}
\end{array}\]

In what follows, we shall define
\[H_\gamma=\le\{ u:\mathbb{R}^2\to \mathbb{R}\,:\,\int_{\mathbb{R}^2}|u(x)|^2\,\gamma^{\vee}(x)\,dx<\infty\r\}=L^2(\mathbb{R}^2,\gamma^{\vee}(x)\,dx),\]
and
\[\bar{H}_\gamma=\le\{ f:\Gamma \to \mathbb{R}\,:\,\sum_{k=1}^n\int_{I_k}|f(z,k)|^2\,\gamma(z,k)\,T_k(z)\,dz<\infty\r\}=L^2(\Gamma,\nu_\gamma),\]
where
 the measure $\nu_\gamma$ is defined as
\[\nu_\gamma(A):= \sum_{k=1}^n \int_{I_k} \mathbb{I}_A(z,k) \,\gamma(z,k)\,T_k(z)\,dz,\ \ \ \ A\subseteq \mathcal{B}(\Gamma).\]

Now, for every $u:\mathbb{R}^2\to \mathbb{R}$ we define
\[u^{\wedge }(z,k)=\frac 1{T_k(z))}\oint_{C_k(z)} \frac{u(x)}{|\nabla H(x)|}\,dl_{z,k}=\oint_{C_k(z)} u(x)\,d\mu_{z,k},\ \ \ \ (z,k) \in\,\Gamma,\]
and for every $f :\Gamma\to \mathbb{R}$ we define
\[f^{\vee}(x)=f(\Pi(x)),\ \ \ \ x \in\,\mathbb{R}^2.\]

\begin{Proposition}
\label{p31}
Assume the Hamiltonian $H$ satisfies Hypothesis \ref{H1} and  $\gamma:\Gamma\to (0,+\infty)$ is a  weight function satisfying condition \eqref{fluid1}. Then, the following holds.
\begin{enumerate}
\item For every $u \in\,H_\gamma$ we have $u^{\wedge} \in\,\bar{H}_\gamma$ and
\begin{equation}
\label{fluid7}
|u^{\wedge}|_{\bar{H}_\gamma}\leq |u|_{H_\gamma}.\end{equation}
\item For every $f \in\,\bar{H}_\gamma$ we have $f^{\vee } \in\,{H}_\gamma$ and
\begin{equation}
\label{fluid8}
|f^{\vee}|_{{H}_\gamma}= |f|_{\bar{H}_\gamma}.\end{equation}
\end{enumerate}
\end{Proposition}

\begin{proof}
Let $u \in\,H_\gamma$. Recalling how the probability measure $\mu_{z,k}$ has been defined in \eqref{brr}, due to \eqref{fluid10} we have 
\[
\begin{array}{l}
\ds{|u^{\wedge}(x)|^2_{\bar{H}_\gamma}=\sum_{k=1}^n \int_{I_k}|u^{\wedge} (z,k)|^2\gamma(z,k) T_k(z)\,dz}\\
\vs
\ds{
=\sum_{k=1}^n \int_{I_k}\le|\oint_{C_k(z)} u(x)\,d\mu_{z,k}\r|^2\gamma(z,k)\,T_k(z)\,dz\leq \sum_{k=1}^n\int_{I_k}\oint_{C_k(z)} |u(x)|^2 d\mu_{z,k}\,\gamma(z,k)\,T_k(z)\,dz}\\
\vs
\ds{=\sum_{k=1}^n \int_{I_k}\oint_{C_k(z)} \frac{|u(x)|^2\gamma^{\vee}(x)}{|\nabla H(x)|}\,dl_{z,k}\,dz=\int_{\mathbb{R}^2}|u(x)|^2\,\gamma^{\vee}(x)\,dx=|u|_{H_\gamma}^2.}
\end{array}\]
This implies that $u^{\wedge} \in\,\bar{H}_{\gamma}$ and \eqref{fluid7} holds.

Concerning the second part of the Proposition, by using again \eqref{fluid10} for every $f \in\,\bar{H}_\gamma$ we have
\[\begin{array}{l}
\ds{ |f^{\vee}|^2_{H_\gamma}=\int_{\mathbb{R}^2}|f^\vee(x)|^2\gamma^{\vee}(x)\,dx=\int_{\mathbb{R}^2}|f(\Pi(x))|^2\gamma(\Pi(x))\,dx}\\
\vs
\ds{=\sum_{k=1}^n \int_{I_k}\oint_{C_k(z)} \frac{|f(z,k)|^2\gamma(z,k)}{|\nabla H(x)|}\,dl_{z,k}\,dz=\sum_{k=1}^n\int_{I_k}|f(z,k)|^2\gamma(z,k)\,T_k(z)\,dz=|f|^2_{\bar{H}_{\gamma}}.}
\end{array}\]
This allows to conclude that $f^{\vee} \in\,H_\gamma$ and \eqref{fluid8} holds.

\end{proof}

For every $u \in\,H_\gamma$ and $f \in\,\bar{H}_\gamma$, we have
\begin{equation}
\label{fluid11}
\le<f,u^\wedge \r>_{\bar{H}_\gamma}=\le<f^\vee, u\r>_{H_\gamma}.
\end{equation}
Actually, we have
\[\begin{array}{l}
\ds{\le<f,u^\wedge \r>_{\bar{H}_\gamma}=\sum_{k=1}^n \int_{I_k} f(z,k) u^\wedge (z,k)\,\gamma(z,k)\,T_k(z)\,dz}\\
\vs
\ds{=\sum_{k=1}^n \int_{I_k} f(z,k)\oint_{C_k(z)}u(x)\,d\mu_{z,k} \,\gamma(z,k)\,T_k(z)\,dz}\\
\vs
\ds{=\sum_{k=1}^n \int_{I_k}\oint_{C_k(z)}\frac{f(\Pi(x)) u(x)\gamma^{\vee}(x)}{|\nabla H(x)|}\,dl_{z,k} \,dz,}\end{array}\]
so that, thanks to \eqref{fluid10}, we can conclude
\[\le<f,u^\wedge \r>_{\bar{H}_\gamma}=\int_{\mathbb{R}^2} f(\Pi(x)) u(x)\gamma^\vee(x)\,dx=\le<f^\vee, u\r>_{H_\gamma}.\]
Moreover, for every $f \in\,\bar{H}_\gamma$ and $u \in\,H_\gamma$, we have 
\[\begin{array}{l}
\ds{\le(f^\vee u\r)^\wedge(z,k)=\oint_{C_k(z)} f^\vee(x)u(x)\,d\mu_{z,k}=\oint_{C_k(z)} f(\Pi(x))u(x)\,d\mu_{z,k}}\\
\vs
\ds{=f(z,k)\,\oint_{C_k(z)} u(x)\,d\mu_{z,k}=f(z,k)u^\wedge (z,k),}
\end{array}\]
so that
\begin{equation}
\label{fluid12}
\le(f^\vee u\r)^\wedge =f u^\wedge.
\end{equation}

As a consequence of \eqref{fluid11} and \eqref{fluid12}, we conclude that for any $f,g \in\,\bar{H}_\gamma$
\[\langle f,g\rangle_{\bar{H}_\gamma}=\langle\le(f^\vee\r)^\wedge,g\rangle_{\bar{H}_\gamma}=\langle f^\vee,g^\vee\rangle_{H_\gamma}.
\]
In particular, if $\{f_n\}_{n \in\, \mathbb{N}}$ is an orthonormal system in $\bar{H}_\gamma$, it follows that the system $\{(f_n)^\vee\}_{n \in\, \mathbb{N}}$ is  orthonormal  in ${H}_\gamma$.

\begin{Lemma}
\label{lf51}
Assume $u \in\,C^1(\mathbb{R}^2)$. Then, under Hypothesis \ref{H1}, the mapping $u^\wedge$ is continuously differentiable with respect to $z$ in the set $\bigcup_{k} \mathring{I}_k$. Moreover, for every fixed $\d>0$ and $M>0$
\begin{equation}
\label{f51}
\sup_{\substack{(z,k) \in\,\bigcup_{h} \le[\mathring{I}_h\cap \Pi(G(\pm \d))^c\r]\\ z\leq M}}\le|\frac{\partial u^\wedge}{\partial z}(z,k)\r|=: c_{\d,M}<\infty.
\end{equation}

\end{Lemma}

\begin{proof}
If $u \in\,C^1(\mathbb{R}^2)$, then the mapping 
\[v(x)=\frac{u(x)}{|\nabla H(x)|^2},\ \ \ \ x \in\, \bigcup_{k} \mathring{I}_k,\]
is continuously differentiable. As we have seen in Subsection \ref{ss2.2}, due to \cite[Lemma 8.1.1]{fw12}, this implies 
that the mapping
\[z \mapsto \oint_{C_k(z)}\frac{u(x)}{|\nabla H(x)|}\,dl_{z,k}=\oint_{C_k(z)} v(x)|\nabla H(x)|\,dl_{z,k}\]
is continuously differentiable for every $z$ such that $(z,k) \in\,\mathring{I}_k$. Moreover,  thanks to \eqref{f50}, for every $\d>0$ the derivative is uniformly bounded with respect to $z$ such that  $(z,k) \in\,\Pi(G(\pm \d))^c\cap \mathring{I}_k$.
In particular, if we take $u=1$, we get that the mapping $z \mapsto T_k(z)$ is continuously differentiable and the derivative is  uniformly bounded with respect to $z$ such that  $(z,k) \in\,\Pi(G(\pm \d))^c\cap \mathring{I}_k$.
Therefore, as 
\[u^\wedge(z,k)=\oint_{C_k(z)} u(x)\,d\mu_{z,k}=\frac 1{T_k(z)}\oint_{C_k(z)} \frac {u(x)}{|\nabla H(x)|}\,dl_{z,k},\]
and $T_k(z)$ remains uniformly bounded from zero for $(z,k) \in\,\Pi(G(\pm \d))^c\cap \mathring{I}_k$, we can conclude.  

\end{proof}

\medskip

For every $Q \in\,\mathcal{L}(H_\gamma)$ and $f \in\,\bar{H}_{\gamma}$, we define
\[Q^\wedge f=(Q f^\vee)^\wedge.\]
Thanks to Proposition \ref{p31}, we have that
\[\|Q^\wedge\|_{\mathcal{L}(\bar{H}_\gamma)}\leq \|Q\|_{\mathcal{L}({H}_\gamma)}.\]
In an analogous way, for any $A \in\,\mathcal{L}(\bar{H}_\gamma)$ and $u \in\,H_\gamma$, we define
\[A^\vee u=(A u^\wedge)^\vee.\]
As above, due to Proposition \ref{p31},  we have that $A^\vee \in\,\mathcal{L}(H_\gamma)$ and 
\[\|A^\vee\|_{\mathcal{L}(H_\gamma)}\leq \|A\|_{\mathcal{L}(\bar{H}_\gamma)}.\]
Moreover, due to \eqref{fluid12}, we can check immediately that $(A^\vee)^\wedge=A$.

\section{The semigroup $S_\e(t)$ in the weighted space $H_\gamma$}
\label{sec3bis}

Since $\text{div}\bar{\nabla} H=0$,  it is immediate to check that the Lebesgue measure is invariant for the semigroup $S_\e(t)$, for every fixed $\e>0$. In particular,  $S_\e(t)$ can be extended to a contraction semigroup on $L^p(\mathbb{R}^p)$, for every $p\geq 1$ and $\e>0$.

In what follows, we shall make the following fundamental assumption about the behavior of $S_\e(t)$ in the weighted space ${H}_\gamma$

\begin{Hypothesis}
\label{H2}
The semigroup $S_\e(t)$ is well defined on $H_\gamma$, for every $\e>0$. Moreover, for every fixed $T>0$, there exists $c_T>0$ such that
\begin{equation}
\label{f65}
\|S_\e(t)\|_{\mathcal{L}(H_\gamma)}\leq c_T,\ \ \ \ \ t \in\,[0,T],\ \ \ \ \e>0.
\end{equation}
\end{Hypothesis}

In next proposition  we will show that, in fact, Hypothesis \ref{H2} is fulfilled in some relevant cases.

\begin{Proposition}
Assume  that the Hamiltonian $H$ satisfies Hypothesis \ref{H1}. Then, there exists  a strictly positive continuous function $\gamma:\Gamma\to (0,+\infty)$,  satisfying \eqref{fluid1},  such that Hypothesis \ref{H2} is verified.
\end{Proposition}

\begin{proof}
According to condition 3. in Hypothesis \ref{H1}, we can fix 
\[z_0> \max_{i=1,\ldots,n} H(x_i)\]
 such that
\[H(x)\geq z_0 \Longrightarrow H(x)\geq a_1\,|x|^2.\]
Once fixed $z_0$, we take a positive decreasing function $h \in\,C^2([0,+\infty))$, such that
\[h(t)=\begin{cases}
\frac 12 \,\exp(-\la \le(\sqrt{t}-\sqrt{2z_0}\r)),  &  t\geq 2z_0,\\
1,  &  t\leq z_0,
\end{cases}\]
for some arbitrary $\la>0$, and define
\[\gamma(z,k)=h(z),\ \ \ \ (z,k) \in\,\Gamma.\]

In view of \eqref{fluid4} and \eqref{fluid5}, it is immediate to check that for every interval $I_k$ not having $O_\infty$ as its end point,
\[\int_{I_k} T_k(z)\,dz<\infty,\]
so that, as $|h(t)|\leq 1$, we get
\begin{equation}
\label{f200}
\int_{I_k} \gamma(z,k) T_k(z)\,dz<\infty.\end{equation}
Moreover, if $I_k\sim O_\infty$, due to \eqref{fluid6}, since $h(t)$ has exponential decay, we have that \eqref{f200} holds as well and hence we can conclude that \eqref{fluid1} holds.

Now, recalling that $\gamma^\vee(x)=\gamma(\Pi(x)),$  we have that
\[\gamma^\vee(x)=h(H(x)),\ \ \ \ x \in\,\mathbb{R}^2.\]
Since $h \in\,C^2([0,+\infty))$, this implies that $\gamma^\vee \in\,C^2(\mathbb{R}^2)$. Moreover,
\[\Delta \gamma^\vee(x)=h^{\prime \prime}(H(x))|\nabla H(x)|^2+h^\prime(H(x))\Delta h(x),\ \ \ \ x \in\,\mathbb{R}^2.\]
Hence, as for any $t\geq  2 z_0$ we have
\[h^\prime(t)=-\frac \la 2\, h(t) t^{-\frac 12},\ \ \ \ h^{\prime \prime}(t)= \frac \la 4\,h(t)\le(\la t^{-1}+t^{-\frac 32}\r),\]
we obtain
\[\begin{array}{l}
\ds{\Delta \gamma^\vee(x)=\frac \la 4\, \gamma^\vee(x)  \le[ \frac{|\nabla H(x)|^2}{H(x)}\le(\la+\frac{1}{H(x)^{\frac 12}}\r)-2\,\frac{\Delta H(x)}{H(x)^{\frac 12}}\r].}
\end{array}\]
Since $|\nabla H(x)|\leq c\,|x|$ and $|\Delta H(x)|\leq c$, for every $x \in\,\mathbb{R}^2$, and $H(x)\geq  a_1\,|x|^2$, if $H(x)\geq z_0$, we have
\[ \sup_{\substack{x \in\,\mathbb{R}^2,\\H(x)\geq z_0}}\le| \frac{|\nabla H(x)|^2}{H(x)}\le(\la+\frac{1}{H(x)^{\frac 12}}\r)-2\,\frac{\Delta H(x)}{H(x)^{\frac 12}}\r|=:\kappa_\la<\infty,\]
so that
\begin{equation}
\label{f66}
H(x)\geq 2 z_0 \Longrightarrow \Delta \gamma^\vee(x)\leq \frac {\la\,\kappa_\la} 4 \gamma^\vee(x). \end{equation}
On the other hand, since $\gamma^\vee \in\,C^2(\mathbb{R}^2)$, we have that 
\[\sup_{\substack{x \in\,\mathbb{R}^2,\\H(x)< 2 z_0}}|\Delta \gamma^\vee(x)|=:M<\infty,\]
and hence 
\[H(x)< 2 z_0 \Longrightarrow \Delta \gamma^\vee(x)\leq M\leq 2\,h(H(x))\,M=2\,\gamma^\vee(x)\,M. \]
Together with \eqref{f66}, this implies that there exists some $c>0$ such that
\begin{equation}
\label{f67}
\Delta \gamma^\vee(x)\leq c\,\gamma^\vee(x),\ \ \ \ x \in\,\mathbb{R}^2.
\end{equation}

The generator of the semigroup $S_\e(t)$ is the second order differential operator
\[\mathcal{L}_\e=\frac 12 \Delta+\frac 1\e\bar{\nabla} H(x)\cdot \nabla.\]
As $\text{div}\,\bar{\nabla} H(x)=0$, it is immediate to check that the adjoint $\mathcal{L}^\star_\e$ of the operator $\mathcal{L}_\e$ in $L^2(\mathbb{R}^2)$ is given by
\[\mathcal{L}_\e^\star=\frac 12 \Delta-\frac 1\e \bar{\nabla} H(x)\cdot \nabla.\]
This implies that the adjoint semigroup $S^\star_\e(t)$ is the Markov transition semigroup associated with the equation
\[dY_\e(t)=-\frac 1\e \bar{\nabla}H(Y_\e(t))\,dt+d\hat{w}(t),\ \ \ \ Y_\e(0)=x,\]
for some $2$d Brownian motion $\hat{w}(t)$.
Now, by using the It\^o formula, we have
\[\begin{array}{l}
\ds{\gamma^\vee(Y_\e(t))=\gamma^\vee(x)-\frac 1\e\int_0^t\langle \nabla\gamma^\vee(Y_\e(s)),\bar{\nabla}H(Y_\e(s))\rangle\,ds}\\
\vs
\ds{+\int_0^t\langle \nabla\gamma^\vee(Y_\e(s)),dw(s)\rangle+\frac 12\int_0^t\Delta\gamma^\vee(Y_\e(s))\,ds,}
\end{array}\]
and then, since 
\[\nabla \gamma^\vee(x)=h^\prime(H(x))\,\nabla H(x),\ \ \ \ \ x \in\,\mathbb{R}^2,\]
it follows
\[\begin{array}{l}
 \ds{\gamma^\vee(Y_\e(t))=\gamma^\vee(x)+\int_0^t\langle \nabla\gamma^\vee(Y_\e(s)),dw(s)\rangle+\frac 12\int_0^t\Delta\gamma^\vee(Y_\e(s))\,ds.}
\end{array}\]
In view of \eqref{f67}, by taking the expectation of both sides above, we get
\[S^\star_\e(t)\gamma^\vee(x)=\hat{\mathbb{E}}_x\gamma^\vee(Y_\e(t))\leq \gamma^\vee(x)+c\int_0^t\hat{\mathbb{E}}_x\gamma^\vee(Y_\e(s))\,ds=\gamma^\vee(x)+c\int_0^tS^\star_\e(s)\gamma^\vee(x)\,ds,\]
and, thanks to the Gronwall lemma, this allows to conclude that 
\begin{equation}
\label{f68}
S^\star_\e(t)\,\gamma^\vee(x)\leq e^{c\,t}\,\gamma^\vee(x),\ \ \ \ x \in\,\mathbb{R}^2,\ \ \ t\geq 0.\end{equation}

Now, for any $u \in\,H_\gamma$, we have
\[\begin{array}{l}
\ds{|S_\e(t)u|^2_{H_\gamma}=\int_{\mathbb{R}^2}|S_\e(t)u(x)|^2\,\gamma^\vee(x)\,dx\leq \int_{\mathbb{R}^2}S_\e(t)|u|^2(x)\,\gamma^\vee(x)\,dx}\\
\vs
\ds{=\langle S_\e(t)|u|^2,\gamma^\vee\rangle_{L^2(\mathbb{R}^2)}=\langle |u|^2,S^\star_\e(t)\gamma^\vee\rangle_{L^2(\mathbb{R}^2)}.}
\end{array}\]
Then, from \eqref{f68}, we have
\[|S_\e(t)u|^2_{H_\gamma}\leq e^{c\,T}\langle |u|^2,\gamma^\vee\rangle_{L^2(\mathbb{R}^2)}=e^{c\,T}|u|^2_{H_\gamma},\ \ \ \ t \in\,[0,T],\]
which implies \eqref{f65}, with $c_T=e^{c T}$.

\end{proof}

\section{The limiting result for the linear deterministic problem}
\label{sec4}

For every $\e>0$, we consider the   linear parabolic Cauchy problem associated with the second order differential operator $\mathcal{L}_\e$
\begin{equation}
\label{fluid15}
\le\{\begin{array}{l}
\ds{\partial_t v_\e(t,x)=\mathcal{L}_\e v_\e(t,x),\ \ \ \ \ t>0,\ \ \ x \in\,\mathbb{R}^2,}\\
\vs
\ds{v_\e(0,x)=g(x),\ \ \ \ x \in\,\mathbb{R}^2.}
\end{array}\r.
\end{equation}

The solution of problem \eqref{fluid15} has a probabilistic representation in terms of the Markov transition semigroup $S_\e(t)$ associated with equation \eqref{ham-epsilon-bis}, that, we recall, is defined for any bounded Borel function $\varphi:\mathbb{R}^2\to \mathbb{R}$ by
 \[S_\e(t) \varphi(x)=\mathbb{E}_{\,x}\, \varphi(X_\e(t)),\ \ \ \ x \in\,\mathbb{R}^2.\]
 Actually, as a consequence of It\^o's formula, if the initial condition $\varphi$ is taken in $C^2_b(\mathbb{R}^2)$, then
\begin{equation}
\label{fluid16}
\frac{d}{dt}[S_\e(\cdot)\varphi(x)](t)=\mathcal{L}_\e\le(S_\e(t)\varphi\r)(x),\ \ \ \ t\geq 0,\ \ \ x \in\,\mathbb{R}^2.
\end{equation}
Moreover, as the Hamiltonian $H$ is assumed to be of class $C^4(\mathbb{R})$, with bounded second derivative, we have that the semigroup $S_\e(t)$ has a smoothing effect, namely it  maps Borel bounded functions into $C^3_b(\mathbb{R}^2)$, for any $t>0$. Thanks to the semigroup law, this allows to  conclude that \eqref{fluid16} is satisfied on $\mathbb{R}^2$, for any Borel bounded  function $\varphi$ and for all $t>0$.

In the present section, we assume that the Hamiltonian $H$, in addition to all conditions in Hypothesis \ref{H1}, satisfies also the following condition.

\begin{Hypothesis}
\label{H3}
It holds
\begin{equation}
\label{fluid18}
\frac{d T_k(z)}{dz}\neq 0,\ \ \ (z,k) \in\,\Gamma.
\end{equation}
\end{Hypothesis}

\begin{Remark}
{\em Condition \eqref{fluid18} rules out the case $H(x)=|x|^2$. This means that with our method we cannot treat the  case $\bar{\nabla} H(x)$ is linear in order to prove the main result stated below in \eqref{fluid17}. 
}
\end{Remark}

Our purpose here is to study the asymptotic behavior of $S_\e(t)\varphi(x)$, and hence of $v_\e(t,x)$, as $\e\downarrow 0$. As a consequence of \eqref{limite-fund}, if the Hamiltonian $H$ satisfies Hypothesis \ref{H1}, then for any continuous mapping $\psi:\Gamma\to\mathbb{R}$ we have
\begin{equation}
\label{fluid25}
\lim_{\e\to 0} S_\e(t) \psi^\vee(x)=\bar{S}(t)\psi(\Pi(x)),\ \ \ \ t\geq 0,\ \ \ x \in\,\mathbb{R}^2.
\end{equation}
As a first thing, we are going to prove that the limit above is uniform with respect to $t \in\,[0,T]$, for every fixed $T>0$. To this purpose, we introduce some notation.

For any $\eta>0$ we  can fix 
\begin{equation}
\label{f33}
z_\eta\geq \max_{i=1,\ldots,n} H(x_i)+1,\end{equation}
(we recall $x_1,\ldots,x_n$ are the critical points of the Hamiltonian $H$), such that, for all $\e>0$ sufficiently small,
\begin{equation}
\label{f31}
\mathbb{P}_x \le(\rho_{\e,\eta}\leq T\r)\leq \eta,\ \ \ \ \ \ \bar{\mathbb{P}}_{\Pi(x)} \le(\rho_\eta\leq T\r)\leq \eta, \end{equation}
 where  
\begin{equation}
\label{f41}
\rho_{\e,\eta}:=\inf\,\{t\geq 0\,:\,H(X_\e^x(t))\geq z_\eta\},\ \ \ \ \ \rho_\eta:=\inf\,\{t\geq 0\,:\,\Pi_1\bar{Y}(t)\geq z_\eta\},
\end{equation}
(here we have  denoted $\pi_1(z,k)=z$). Actually, as proved in \cite[Lemma 3.2]{fw12} the family
$\{\Pi(X_\e(\cdot))\}_{\e \in\,(0, \e_0)}$ is tight in $C([0,T];\mathbb{R})$. This means that there exists some $z_\eta>0$ as in \eqref{f33} such that
\[\mathbb{P}_x\le(\sup_{t\leq T}H(X_\e(t))\geq z_\eta\r)\leq \eta,\]
and \eqref{f31} follows.
 
\begin{Proposition}
For every $T>0$ and $\psi \in\,C_b(\Gamma)$, we have
\begin{equation}
\label{f220}
\lim_{\e\to 0}\sup_{t \in\,[0,T]}\,\le|\mathbb{E}_x\psi^\vee(X_\e(t))-\bar{\mathbb{E}}_{\Pi(x)} \psi(\bar{Y}(t))\r|=0.
\end{equation}
\end{Proposition}
\begin{proof}
For every $\e>0$, let us define
\[f_\e(t):=\mathbb{E}_x\psi^\vee(X_\e(t))-\bar{\mathbb{E}}_{\Pi(x)} \psi(\bar{Y}(t)).\]
For every fixed $t\geq 0$, we have  
\begin{equation}
\label{f202}
\lim_{\e\to 0} f_\e(t)=0.
\end{equation}
If we prove that the family of functions $\{f_\e\}_{\e>0}$ is equibounded and equicontinuous in $C([0,T])$,  by the Ascoli-Arzel\`a Theorem we have that the limit in \eqref{f202} is uniform.

Now, the equiboundedness of $\{f_\e\}_{\e>0}$ follows from the fact that $\psi$ is bounded. In order to prove the equicontinuity of  $\{f_\e\}_{\e>0}$, first of all we notice that we may assume that both $\psi$ and $\psi^\vee$ are uniformly continuous. Actually, due to \eqref{f31} for every $\eta>0$ we have
\[\begin{array}{l}
\ds{\le|\mathbb{E}_x\,\psi^\vee(X_\e(t))-\bar{\mathbb{E}}_{\Pi(x)}\psi(\bar{Y}(t))\r|\leq \le|\mathbb{E}_x\,\le(\psi^\vee(X_\e(t))\,;\,\rho_{\e,\eta}\leq T\r)\r|+\le|\bar{\mathbb{E}}_{\Pi(x)}\le(\psi(\bar{Y}(t))\,;\,\rho_\eta\leq T\r)\r|}\\
\vs
\ds{+\le|\mathbb{E}_x\,\le(\psi^\vee(X_\e(t))\,;\,\rho_{\e,\eta}> T\r)-\bar{\mathbb{E}}_{\Pi(x)}\le(\psi(\bar{Y}(t))\,;\,\rho_\eta> T\r)\r|   }\\
\vs
\ds{\leq 2\,\|\psi\|_\infty \,\eta+\le|\mathbb{E}_x\,\le(\psi^\vee(X_\e(t))\,;\,\rho_{\e,\eta}> T\r)-\bar{\mathbb{E}}_{\Pi(x)}\le(\psi(\bar{Y}(t))\,;\,\rho_\eta> T\r)\r| . }
\end{array}\]
Therefore, as $\psi^\vee$ is uniformly continuous on $\{H(x)\leq z_\eta\}$ and $\psi$ is uniformly continuous on $\{z\leq z_\eta\}$, due to the arbitrariness of $\eta$ we can conclude.

Now, if $\psi$ are uniformly continuous, for every $\eta>0$ there exists $\d_\eta>0$ such that
\begin{equation}
\label{f205}
\begin{array}{l}
\ds{|\Pi(x_1)-\Pi(x_2)<\d_\eta,\ \ \ \text{dist}\le((z_1,k_1), (z_2,k_2)\r)<\d_\eta}\\
\vs
\ds{\Longrightarrow |\psi^\vee(x_1)-\psi^\vee(x_2)|+|\psi(z_1,k_1)-\psi(z_2,k_2)|<\frac \eta 2.}
\end{array}
\end{equation}
Moreover, as $\{\Pi(X_\e)\}_{\e>0}\subset C([0,T];\Gamma)$ is tight, there exists a compact set $K_\eta \subset C([0,T];\Gamma)$, such that 
\begin{equation}
\label{f206}
\mathbb{P}_x\le(\Pi(X_\e) \in\,K_\eta^c\r)+\bar{\mathbb{P}}_{\Pi(x)}\le(\bar{Y} \in\,K_\eta^c\r)\leq \frac \eta{4\|\psi\|_\infty}.
\end{equation}
Since functions in $K_\eta$ are equicontinuous, there exists $\theta_\eta>0$ such that on $K_\eta$
\[|t-s|<\theta_\eta\Longrightarrow |\Pi(X_\e(t))-\Pi(X_\e(s))|<\d_\eta,\ \ \ \ |\bar{Y}(t)-\bar{Y}(s)|<\d_\eta.\]
Therefore, thanks to \eqref{f205} and \eqref{f206}, for every $t,s \in\,[0,T]$ such that $|t-s|<\theta_\eta$
\[\begin{array}{l}
\ds{|f_\e(t)-f_\e(s)|}\\
\vs
\ds{\leq \le|\mathbb{E}_x\le(\psi^\vee(X_\e(t))-\psi^\vee(X_\e(s))\,;\,\Pi(X_\e) \in\,K_\eta\r)\r|+\le|\bar{\mathbb{E}}_{\Pi(x)}\le(\psi(\bar{Y}(t))-\psi(\bar{Y}(s))\,;\,\bar{Y} \in\,K_\eta\r)\r|}\\
\vs
\ds{+2\,\|\psi\|_\infty\le(\mathbb{P}_x\le(\Pi(X_\e) \in\,K_\eta^c\r)+\bar{\mathbb{P}}_{\Pi(x)}\le(\bar{Y} \in\,K_\eta^c\r)\r)\leq \frac \eta 2+2\,\|\psi\|_\infty\frac \eta{4\|\psi\|_\infty}=\eta.}
\end{array}\]

\end{proof}

In what follows, we want to show that, in fact,  under suitable conditions,  limit \eqref{f220}  is also true if $\psi^\vee$ is replaced by $u:\mathbb{R}^2\to \reals$ and $\psi$ is replaced by $u^\wedge$. Namely, we want to prove  the following result.

\begin{Theorem}
\label{fluid-main}
Assume that the Hamiltonian $H$ satisfies all conditions in Hypotheses \ref{H1} and \ref{H3}. Then, for any $u \in\,C_b(\mathbb{R}^2)$ and $x \in\,\mathbb{R}^2$, and for any $0<\tau\leq T$, we have
\begin{equation}
\label{fluid17}
\lim_{\e\to 0} \sup_{t \in\,[\tau,T]}\le|S_\e(t)u(x)-\bar{S}(t)^\vee u(x)\r|=\lim_{\e\to 0} \sup_{t \in\,[\tau,T]}\le|\mathbb{E}_x u(X_\e(t))-\bar{\mathbb{E}}_{\Pi(x)} u^\wedge (\bar{Y}(t))\r|=0.
\end{equation}

\end{Theorem}

\subsection{A preliminary result}

 Before proceeding with the proof of Theorem \ref{fluid-main} and of all required preliminary results, we introduce some notations. 
For every $\e, \eta>0$  and $0<\d^\prime<\delta$, by using the notations introduced in Subsection \ref{ss2.2} we define 
\begin{equation}
\label{fluid21}
\begin{array}{l}
\ds{\si_n^{\e,\eta,\d,\d^\prime}=\min\,\le\{t\geq  \tau_n^{\e,\eta,\d,\d^\prime}\ :\ X_\e(t) \in\,G(\pm \d)^c\r\},}\\
\vs
\ds{\tau_n^{\e,\eta,\d,\d^\prime}=\min\,\le\{t\geq  \si_{n-1}^{\e,\eta,\d,\d^\prime}\ :\ X_\e(t) \in\,D(\pm \d^\prime)\cup C(z_\eta)\r\},}
\end{array}
\end{equation}
with $\tau_0^{\e,\eta,\d,\d^\prime}=0$ and $z_\eta$ defined as in \eqref{f33}.
Clearly, after the process $X_\e(t)$ reaches $C(z_\eta)$, all $\tau_n^{\e,\eta,\d,\d^\prime}$ and $\si_n^{\e,\eta,\d,\d^\prime}$ coincide with the stopping time $\rho_{\e,\eta}$ introduced in \eqref{f41},
and for any $n \in\,\nat$
\begin{equation}
\label{f30}
X_\e(\si_n^{\e,\eta,\d,\d^\prime}) \in\,D(\pm \d)\cup C(z_\eta),\ \ \ \ \ X_\e(\tau_n^{\e,\eta,\d,\d^\prime}) \in\,D(\pm \d^\prime)\cup C(z_\eta).
\end{equation}
Moreover, if $X_\e(0) \in\,G(\pm \d)^c$, we have that $\si_0^{\e,\eta,\d,\d^\prime}=0$ and $\tau_1^{\e,\eta,\d,\d^\prime}$ is the first time the process $X_\e$ touches $D(\pm \d^\prime)\cup C(z_\eta)$. In particular, if $X_\e(0)\geq z_\eta$, then $\tau_1^{\e,\eta,\d,\d^\prime}$ is the first time the process $X_\e$ touches $C(z_\eta)$ and all successive stopping times coincide with $\rho_{\e,\eta}$.

\begin{Lemma}
\label{fluid23}
Assume that the same assumptions of Theorem \ref{fluid-main} are verified. Then, for every $u \in\,C_b(\mathbb{R}^2)$ and $x \in\,\mathbb{R}^2$, and for every $0<\tau<T$
\begin{equation}
\label{fluid24}
\lim_{\e\to 0}\,\sup_{t \in\,[\tau,T]}\le|\mathbb{E}_x\,(u^\wedge)^\vee(X_\e(t))-\bar{\mathbb{E}}_{\Pi(x)}\,u^\wedge (\bar{Y}(t))\r|=0.
\end{equation}

\end{Lemma}

\begin{proof}
Thanks to \eqref{f220}, if $u^\wedge$ were a continuous function on $\Gamma$, then \eqref{fluid24} would follow immediately. Unfortunately, because of the presence of the interior vertices, even if $u$ is continuous on $\mathbb{R}^2$ we cannot conclude that $u^\wedge$ is continuous on $\Gamma$, in general. This means that we have to treat separately the internal vertices and the rest of the points of the graph $\Gamma$.

First of all, we fix $\d>0$ and $f_\d \in\,C_b(\Gamma)$ such that
\[\|f_\d\|_\infty\leq \|u^\wedge\|_\infty\leq \|u\|_\infty,\ \ \ \ f_\d=u^\wedge,\ \text{on}\ \Pi(G(\pm \d/2)^c).\]
Thus, for every $\d \in\,(0,1)$  and $t\geq 0$, we can write
\[\begin{array}{l}
\ds{\mathbb{E}_x\,(u^\wedge)^\vee(X_\e(t))-\bar{\mathbb{E}}_{\Pi(x)}\,u^\wedge (\bar{Y}(t))=\mathbb{E}_x\,\le[(u^\wedge)^\vee(X_\e(t))-f_\d^\vee(X_\e(t))\r]}\\
\vs
\ds{+\le[\mathbb{E}_x\,f_\d^\vee(X_\e(t))-\bar{\mathbb{E}}_{\Pi(x)}\,f_\d (\bar{Y}(t))\r]+\bar{\mathbb{E}}_{\Pi(x)}\le[f_\d (\bar{Y}(t))-u^\wedge (\bar{Y}(t))\r]=:I_1^{\e,\d}(t)+I_2^{\e,\d}(t)+I^\d(t).}
\end{array}\]
If we prove that for any $\eta>0$ there exist $\d_\eta, \e_\eta >0$ such that
\begin{equation}
\label{f26}
\sup_{t \in\,[\tau,T]}\le(|I^{\e,\d_\eta}_1(t)|+|I^{\d_\eta}(t)|\r)<\eta,\ \ \ \ \ \e<\e_\eta,
\end{equation}
then
\[\begin{array}{l}
\ds{\sup_{t \in\,[\tau,T]}\le|\mathbb{E}_x\,(u^\wedge)^\vee(X_\e(t))-\bar{\mathbb{E}}_{\Pi(x)}\,u^\wedge (\bar{Y}(t))\r|}\\
\vs
\ds{\leq \eta+\sup_{t \in\,[\tau,T]}\le|\mathbb{E}_x\,f_{\d_\eta}^\vee(X_\e(t))-\bar{\mathbb{E}}_{\Pi(x)}\,f_{\d_\eta} (\bar{Y}(t))\r|,\ \ \ \ \e<\e_\eta.}
\end{array}\]
Since $f_{\d_\eta} \in\,C_b(\Gamma)$, due to \eqref{f220} and the arbitrariness of $\eta>0$, this implies  \eqref{fluid24}.

Thus, let us prove \eqref{f26}. If we take $\eta^\prime=\eta/8\|u \|_\infty$, due to \eqref{f31} we have
\begin{equation}
\label{f36}
\begin{array}{l}
\ds{I_1^{\e,\d}(t)}\\
\vs
\ds{=\mathbb{E}_x\,\le( (u^\wedge)^\vee(X_\e(t))-f_\d^\vee(X_\e(t))\,;\,t<\rho_{\e,\eta^\prime}\r)+\mathbb{E}_x\,\le( (u^\wedge)^\vee(X_\e(t))-f_\d^\vee(X_\e(t))\,;\,t\geq \rho_{\e,\eta^\prime}\r)}\\
\vs
\ds{\leq \mathbb{E}_x\,\le( (u^\wedge)^\vee(X_\e(t))-f_\d^\vee(X_\e(t))\,;\,t<\rho_{\e,\eta^\prime}\r)+2\,\|u \|_{\infty}\,\mathbb{P}_x( t\geq \rho_{\e,\eta^\prime})}\\
\vs
\ds{\leq \mathbb{E}_x\,\le( (u^\wedge)^\vee(X_\e(t))-f_\d^\vee(X_\e(t))\,;\,t<\rho_{\e,\eta^\prime}\r)+\frac \eta 4=:J^{\e,\eta^\prime,\d}_1(t)+\frac \eta 4.}
\end{array}\end{equation}

Recalling that $f_\d=u^\wedge$ on $\Pi(G(\pm \d/2))^c$, we have that $f_\d^\vee =(u^\wedge)^\vee$ on $G(\pm \d/2)^c$, so that 
\begin{equation}
\label{f37}
\begin{array}{l}
\ds{J_1^{\e,\eta^\prime,\d}(t)=\mathbb{E}_x\,\le( (u^\wedge)^\vee(X_\e(t))-f_\d^\vee(X_\e(t))\,;\,t \in\,\bigcup_{n=0}^\infty \le[\tau_n^{\e,\eta^\prime,\d,\d/2},\si_n^{\e,\eta^\prime,\d,\d/2}\r)\r)}\\
\vs
\ds{= \sum_{n=0}^\infty\, \mathbb{E}_x\,\le( (u^\wedge)^\vee(X_\e(t))-f_\d^\vee(X_\e(t))\,;\,t \in\,\le[\tau_n^{\e,\eta^\prime,\d,\d/2},\si_n^{\e,\eta^\prime,\d,\d/2}\r)\r)=:\sum_{n=0}^\infty J_{1,n}^{\e,\eta^\prime,\d}(t).}
\end{array}\end{equation}
Due to the strong Markov property, we have
\[\begin{array}{l}
\ds{\le|J_{1,n}^{\e,\eta^\prime,\d}(t)\r|=\le|
\mathbb{E}_x\,\le( \le[(u^\wedge)^\vee(X_\e(t))-f_\d^\vee(X_\e(t))\r]\,
\mathbb{I}_{\{t\geq \tau_n^{\e,\eta^\prime,\d,\d/2}\}}\,\mathbb{I}_{\{\si_n^{\e,\eta^\prime,\d,\d/2}> t\}}\r)\r|}\\
\vs
\ds{
\leq \mathbb{E}_x\,\le( \mathbb{I}_{\{\tau^{\e,\eta^\prime,\d,\d/2}_n\leq t\}}\le|\mathbb{E}_{X_\e(\tau_n^{\e,\eta^\prime,\d,\d/2})}\le((u^\wedge)^\vee(X_\e(t))-f_\d^\vee(X_\e(t))\,;\,t <\si_0^{\e,\eta^\prime,\d,\d/2}\r)\r|\r).}
\end{array}\]
Thanks to \eqref{f30}, this implies
\begin{equation}
\label{f32}
\begin{array}{l}
\ds{|J_{1,n}^{\e,\eta^\prime,\d}(t)|}\\
\vs
\ds{\leq \mathbb{P}_x\le(\tau^{\e,\eta^\prime,\d,\d/2}_n\leq t\r)\sup_{y \in\,D(\pm \d/2)\cup C(z_\eta^\prime)}\le|\mathbb{E}_y\le((u^\wedge)^\vee(X_\e(t))-f_\d^\vee(X_\e(t))\,;\,t <\si_0^{\e,\eta^\prime,\d,\d/2}\r)\r|
}\\
\vs
\ds{\leq e^t\,\mathbb{E}_x\,e^{-\tau^{\e,\eta^\prime,\d,\d/2}_n} \sup_{y \in\,D(\pm \d/2)\cup C(z_{\eta^\prime})}\le|\mathbb{E}_y\le((u^\wedge)^\vee(X_\e(t))-f_\d^\vee(X_\e(t))\,;\,t <\si_0^{\e,\eta^\prime,\d,\d/2}\r)\r|.}
\end{array}\end{equation}

According to what proved in \cite[Section 8.3, see (8.3.14)]{fw12}, there exist a constant $c>0$ and  $\d_1>0$  such that for all $x \in\,\mathbb{R}^2$ and for all $\d\leq \d_1$ and $\e>0$ sufficiently small   it holds
\begin{equation}
\label{f46}
\begin{array}{l}
\ds{\sum_{n=0}^\infty \mathbb{E}_x\,e^{-\tau^{\e,\eta^\prime,\d,\d/2}_n}=\sum_{n=0}^\infty \mathbb{E}_x\,\le(e^{-\tau^{\e,\eta^\prime,\d,\d/2}_n}\,;\,\tau^{\e,\eta^\prime,\d,\d/2}_n\leq \rho_{\e,\eta^\prime}\r)\leq \frac c\d.}
\end{array}
\end{equation}
Therefore, from \eqref{f32} we get
\begin{equation}
\label{f35}
\sum_{n=0}^\infty |J_{1,n}^{\e,\eta^\prime,\d}(t)|\leq \frac c\d\,\frac{2\|u \|_\infty e^t}t\,\sup_{y \in\,D(\pm \d/2)}\mathbb{E}_y\,\si_0^{\e,\eta^\prime,\d,\d/2}.
\end{equation}
Because of our definition, $\si_0^{\e,\eta^\prime,\d,\d/2}$ is the first exit time of the process $X_\e(t)$ from $G(\pm \d)$. In \cite[Section 8.5, (8.5.17)]{fw12} it is proved that there exists $\d_2>0$ such that for all $\d<\d_2$ and $\e>0$ sufficiently small
\begin{equation}
\label{f43}
\mathbb{E}_y\,\si_0^{\e,\eta^\prime,\d,\d/2}\leq c\,\d^2|\log \d|,\ \ \ \ y \in\,G(\pm \d).
\end{equation}
In particular, due to  \eqref{f36}, \eqref{f37} and \eqref{f35}, this implies that for all  $\d<\d_0:=\d_1\wedge \d_2$ and all  $\e$ small enough 
\[\sup_{t \in\,[\tau,T]}|I_1^{\e,\d}(t)|\leq \frac c\d\,\frac{2\|u \|_\infty e^T}\tau \d^2|\log \d|+\frac \eta 4.\]
This means that for any $\eta>0$ fixed, there exist $\e_{1,\eta}>0$ and $\d_{1,\eta}>0$ such that
\begin{equation}
\label{f38}
\sup_{t \in\,[\tau,T]}|I_1^{\e,\d}(t)|\leq \frac \eta 2,\ \ \ \e< \e_{1,\eta},\ \ \d< \d_{1,\eta}.
\end{equation}

Concerning $I^\d(t)$, recalling how $f_\d$ was defined, for every $\d>0$ we have
\begin{equation}
\label{f39}
\begin{array}{l}
\ds{|I^\d(t)|\leq 2\,\|u\|_\infty\, \bar{\mathbb{P}}_{\Pi(x)}\le(\bar{Y}(t) \in\,\Pi(G(\pm \d/2))\r)\leq 2\,\|u\|_\infty\, \bar{\mathbb{E}}_{\Pi(x)}\,\psi_\d(\bar{Y}(t))}\\
\vs
\ds{\leq 2\,\|u\|_\infty\,\le| \bar{\mathbb{E}}_{\Pi(x)}\,\psi_\d(\bar{Y}(t))-\mathbb{E}_x\psi_\d^\vee(X_\e(t))\r|+2\,\|u\|_\infty {\mathbb{E}}_{x}\,\psi_\d^\vee(X_\e(t)),}
\end{array}\end{equation}
where $\psi_\d$ is a function in $C_b(\Gamma)$ such that  
\[\mathbb{I}_{\Pi(G(\pm \d/2))}\leq \psi_\d\leq 1,\ \ \ \ \psi_\d\equiv 0,\ \ \text{on}\ \Pi(G(\pm(\d))^c.\]

Now, by proceeding as in the proof of \eqref{f38}, we can find $\e_{2,\eta}>0$ and $\d_{2,\eta}>0$ such that
\begin{equation}
\label{f40}
2\,\|u\|_\infty \sup_{t \in\,[\tau,T]}\,{\mathbb{E}}_{x}\,\psi_\d^\vee(X_\e(t))\leq \frac \eta 4,\ \ \ \e\leq \e_{2,\eta},\ \ \d\leq \d_{2,\eta}.
\end{equation}
Therefore, if we set $\d_\eta:=\d_{1,\eta}\wedge \d_{2,\eta}$, from \eqref{f38} and  \eqref{f39}  we get 
\[\sup_{t \in\,[\tau,T]} \le(|I^{\e,\d_\eta}_1(t)|+|I^{\d_\eta}(t)|\r)\leq \frac \eta 2+2\,\|u\|_\infty \sup_{t \in\,[\tau,T]}\le|\bar{\mathbb{E}}_{\Pi(x)}\,\psi_{\d_\eta}(\bar{Y}(t))-\mathbb{E}_x \psi_{\d_\eta}^\vee(X_\e(t))\r|,\]
for every $\e\leq \e_{1,\eta}\wedge \e_{2,\eta}.$
As $\psi_{\d_\eta} \in\,C_b(\Gamma)$, due to \eqref{f220}, this implies that there exists $\e_\eta\leq \e_{1,\eta}\wedge \e_{2,\eta}$ such that \eqref{f26} holds and hence \eqref{fluid24} follows.

\end{proof}

\subsection{ Proof of Theorem \ref{fluid-main}}
 
 In Lemma \ref{fluid23} we have proved that for any $u \in\,C_b(\mathbb{R}^2)$ and $x \in\,\mathbb{R}^2$ and for any $0<\tau<T$
\[ \lim_{\e\to 0}\,\sup_{t \in\,[\tau,T]}\le|\mathbb{E}_x\,(u^\wedge)^\vee(X_\e(t))-\bar{\mathbb{E}}_{\Pi(x)}\,u^\wedge (\bar{Y}(t))\r|=0.\]
Thus, in order to prove \eqref{fluid17}, it is sufficient to prove that
\begin{equation}
\label{f40}
\lim_{\e\to 0} \sup_{t \in\,[\tau,T]}\le|\mathbb{E}_x \le[u(X_\e(t))-(u^\wedge)^\vee(X_\e(t))\r]\r|=0.
\end{equation}

In what follows, we shall assume that $ u \in\,C^1_b(\mathbb{R}^2)$. Actually, if this is not the case, we can fix a sequence $\{u_n\}_{n \in\,\nat} \subset C^1_b(\mathbb{R}^2)$ such that
\[\lim_{n\to \infty} \|u-u_n\|_\infty=0,\ \ \ \ \|u_n\|_\infty\leq \|u\|_\infty.\]
Since this also implies that
\[\lim_{n\to \infty} \|(u^\wedge)^\vee-(u_n^\wedge)^\vee\|_\infty=0,\]
we get
\[\begin{array}{l}
\ds{\lim_{n\to\infty}\  \sup_{\e>0}\,\sup_{t \in\,[\tau,T]}\le|\mathbb{E}_x \le[u(X_\e(t))-u_n(X_\e(t))(X_\e(t))\r]\r|}\\
\vs
\ds{=\lim_{n\to \infty}\ \sup_{\e>0}\,\sup_{t \in\,[\tau,T]} \le|\mathbb{E}_x \le[(u_n^\wedge)^\vee(X_\e(t))-(u^\wedge)^\vee(X_\e(t))\r]\r|=0.}
\end{array}\]
Therefore, in order to prove \eqref{f40}, we have to prove that for any fixed $n \in\,\mathbb{N}$
\[\lim_{\e\to 0} \sup_{t \in\,[\tau,T]}\le|\mathbb{E}_x \le[u_n(X_\e(t))-(u_n^\wedge)^\vee(X_\e(t))\r]\r|=0.\]

To this purpose, let us  fix $\a>0$ and take $\e>0$ small enough so that $\tau-\e^\a>0$.
If we fix  $\eta >0$ and take $\eta^\prime =\eta/4\|u\|_\infty$, we have 
\[\sup_{t\geq 0}\le|\mathbb{E}_x \le(u(X_\e(t))-(u^\wedge)^\vee(X_\e(t))\,;\,t-\e^\a\geq \rho_{\e,\eta^\prime}\r)\r|\leq \frac \eta 2,\ \ \ \ \e>0,\]
where $\rho_{\e,\eta^\prime}$ is the stopping time defined in \eqref{f41} and satisfying \eqref{f31}.
This implies
\begin{equation}
\label{f61}
\begin{array}{l}
\ds{\le|\mathbb{E}_x \le(u(X_\e(t))-(u^\wedge)^\vee(X_\e(t))\r)\r|\leq \le|\mathbb{E}_x \le(u(X_\e(t))-(u^\wedge)^\vee(X_\e(t))\,;\,t-\e^\a<\rho_{\e,\eta^\prime}\r)\r|+\frac \eta 2.}
\end{array}\end{equation}

Now, as in the proof of Lemma \ref{fluid23}, we have
\begin{equation}
\label{f60}
\begin{array}{l}
\ds{\mathbb{E}_x \le(u(X_\e(t))-(u^\wedge)^\vee(X_\e(t))\,;\,t-\e^\a<\rho_{\e,\eta^\prime}\r)}\\
\vs
\ds{=\sum_{n \in\,\mathbb{N}}\,\mathbb{E}_x \le(u(X_\e(t))-(u^\wedge)^\vee(X_\e(t))\,;\,t-\e^\a \in\,[\tau_n^{\e,\eta^\prime,\d,\d/2},\si_n^{\e,\eta^\prime,\d,\d/2})\r) }\\
\vs
\ds{+\sum_{n \in\,\mathbb{N}}\,\mathbb{E}_x \le(u(X_\e(t))-(u^\wedge)^\vee(X_\e(t))\,;\,t-\e^\a \in\, [\si_n^{\e,\eta^\prime,\d,\d/2},\tau_{n+1}^{\e,\eta^\prime,\d,\d/2})\r) }\\
\vs
\ds{=:\sum_{n \in\,\mathbb{N}}\,J_{1,n}^{\e,\eta^\prime,\d}(t)+\sum_{n \in\,\mathbb{N}}\,J_{2,n}^{\e,\eta^\prime,\d}(t).}
\end{array}\end{equation}

As in the proof of Lemma \ref{fluid23}, due to \eqref{f35}, we have
that there exist  $\d_1>0$ and a constant $c>0$ such that for all $\e$ sufficiently small and $\d< \d_1$
\[\sum_{n \in\,\mathbb{N}}\,J_{1,n}^{\e,\eta^\prime,\d}(t)\leq \frac c\d\,\frac{2\|u\|_\infty  e^{t-\e^\a}}{t-\e^\a}\,\sup_{y \in\,D(\pm \d/2)}\mathbb{E}_y\,\si_0^{\e,\eta^\prime,\d,\d/2},\]
so that, thanks to \eqref{f43}, we get
\begin{equation}
\label{f44}
\sup_{t \in\,[\tau,T]}\,\sum_{n \in\,\mathbb{N}}\,J_{1,n}^{\e,\eta^\prime,\d}(t)\leq c\,\frac{\|u\|_\infty e^T}{\tau-\e^\a}\,\d|\log \d|.\end{equation}

On the other  hand, by using once more the strong Markov property, we have
\[\begin{array}{l}
\ds{|J_{2,n}^{\e,\eta^\prime,\d}(t)|}\\
\vs
\ds{\leq \mathbb{E}_x \le(\mathbb{I}_{\{\si_n^{\e,\eta^\prime,\d,\d/2}\leq t-\e^\a\}}\le|\mathbb{E}_{X_\e(\si_n^{\e,\eta^\prime,\d,\d/2})}\le(u(X_\e(t))-(u^\wedge)^\vee(X_\e(t))\,;\,t-\e^\a <\tau_{1}^{\e,\eta^\prime,\d,\d/2}\r)\r|\r)}\\
\vs
\ds{\leq \mathbb{P}_x\le(\si_n^{\e,\eta^\prime,\d,\d/2}\leq t-\e^\a\r)\sup_{y \in\,D(\pm \d)}\,\le|\mathbb{E}_{y}\le(u(X_\e(t))-(u^\wedge)^\vee(X_\e(t))\,;\,t -\e^\a<\tau_{1}^{\e,\eta^\prime,\d,\d/2}\r)\r|}\\
\vs
\ds{\leq e^{t-\e^\a}\, \mathbb{E}_x\le(e^{-\tau_n^{\e,\eta^\prime,\d,\d/2}}\r)\sup_{y \in\,D(\pm \d)}\le|\mathbb{E}_{y}\le(u(X_\e(t))-(u^\wedge)^\vee(X_\e(t))\,;\,t-\e^\a <\tau_{1}^{\e,\eta^\prime,\d,\d/2}\r)\r|,}
\end{array}\]
and thanks to \eqref{f46}, this implies that there exists $\d_2>0$, such that for any $\d\leq \d_2$
\[\sum_{n=1}^\infty |J_{2,n}^{\e,\eta^\prime,\d}(t)|\leq \frac{c\, e^{t-\e^\a}}{\d}\sup_{y \in\,D(\pm \d)}\le|\mathbb{E}_{y}\le(u(X_\e(t))-(u^\wedge)^\vee(X_\e(t))\,;\,t -\e^\a<\tau_{1}^{\e,\eta^\prime,\d,\d/2}\r)\r|.\]
For every $y \in\,\mathbb{R}^2$, we have
\[\begin{array}{l}
\ds{\mathbb{E}_{y}\le(u(X_\e(t))-(u^\wedge)^\vee(X_\e(t))\,;\,t -\e^\a<\tau_{1}^{\e,\eta^\prime,\d,\d/2}\r)}\\
\vs
\ds{=\mathbb{E}_{y}\le(u(X_\e(t))-(u^\wedge)^\vee(X_\e(t-\e^\a))\,;\,t -\e^\a<\tau_{1}^{\e,\eta^\prime,\d,\d/2}\r)}\\
\vs
\ds{+\mathbb{E}_{y}\le((u^\wedge)^\vee(X_\e(t-\e^\a))-(u^\wedge)^\vee(X_\e(t))\,;\,t -\e^\a<\tau_{1}^{\e,\eta^\prime,\d,\d/2},\,t <\tau_{1}^{\e,\eta^\prime,\d,\d/4}\r)}\\
\vs
\ds{+\mathbb{E}_{y}\le((u^\wedge)^\vee(X_\e(t-\e^\a))-(u^\wedge)^\vee(X_\e(t))\,;\,t -\e^\a<\tau_{1}^{\e,\eta^\prime,\d,\d/2},\,t \geq \tau_{1}^{\e,\eta^\prime,\d,\d/4}\r)=:\sum_{i=1}^3 L_i^{\e,\d}(t,y).}\end{array}\]

Let us start considering $L_2^{\e,\d}(t)$.  If $y \in\,D(\pm \d)$, then $\tau_{1}^{\e,\eta^\prime,\d,\d/4}$ is the first time the process $X_\e$ touches $D(\pm \d/4)\cup C(z_{\eta^\prime})$. This means that  
\[t <\tau_{1}^{\e,\eta^\prime,\d,\d/4}\Longrightarrow X_\e(s) \in\,G(\pm \d/4)^c\cap \{H\leq z_{\eta^\prime}\},\ \ \ \ s\leq t.\]
 In particular, $\Pi(X_\e(s))$ remains in the interior of the same edge of the graph $\Gamma$ where $y$ is, for all $s\leq t<\tau_{1}^{\e,\eta^\prime,\d,\d/4}$. As we are assuming that $u \in\,C^1_b(\mathbb{R}^2)$, due to Lemma \ref{lf51} we have that $u^\wedge$ is continuously differentiable on $\Pi(G(\pm \d/4))^c$, with uniformly bounded derivative. 
  
 As a consequence of It\^o's formula, for every $s<t$ we have
 \[\begin{array}{l}
 \ds{H(X_\e(t))-H(X_\e(s))=\frac 12 \int_s^t \Delta H(X_\e(r))\,dr+\int_s^t \langle \nabla H(X_\e(r)), dw(r)\rangle.}
 \end{array}\]
Hence, since for $y \in\,D(\pm \d)$ and $s<t<\tau_{1}^{\e,\eta^\prime,\d,\d/4}$, the process $\Pi(X_\e(s))$ remains in the same edge of the graph $\Gamma$, we get 
 \begin{equation}
 \label{f56}
 \begin{array}{l}
 \ds{\mathbb{E}_y\le(\le|\Pi(X_\e(t))-\Pi(X_\e(s))\r|^2\,;\,t <\tau_{1}^{\e,\eta^\prime,\d,\d/4}\r)=\mathbb{E}_y\le(\le|H(X_\e(t))-H(X_\e(s))\r|^2\,;\,t <\tau_{1}^{\e,\eta^\prime,\d,\d/4}\r)}\\
 \vs
 \ds{\leq c\,\|\Delta H\|^2_\infty (t-s)^2+\sup_{y^\prime \in  \{H\leq z_{\eta^\prime}\}}|\nabla H(y^\prime)|(t-s)\leq c_{T,\eta^\prime} (t-s).}
 \end{array}\end{equation}
 In particular, for every $y \in\,D(\pm \d)$
 \begin{equation}
 \label{f55}
 \begin{array}{l}
 \ds{|L_2^{\e,\d}(t,y)|}\\
 \vs
 \ds{\leq \mathbb{E}_{y}\le(\|u^\wedge\|_{C^1(\Pi(G(\pm \d/4))^c)}|\Pi(X_\e(t-\e^\a)))-\Pi(X_\e(t))|\,;\,t -\e^\a<\tau_{1}^{\e,\eta^\prime,\d,\d/2},\,t <\tau_{1}^{\e,\eta^\prime,\d,\d/4}\r)}\\
 \vs
 \ds{\leq \|u^\wedge\|_{C^1(\Pi(G(\pm \d/4))^c)}\,c_{T,\eta^\prime}^{1/2}\,\e^{\a/2}.}
 \end{array}
 \end{equation}
 
Next, concerning $L_3^{\e,\d}(t,y)$, we have
\begin{equation}
\label{f400}
\begin{array}{l}
\ds{\mathbb{P}_y\le(t -\e^\a<\tau_{1}^{\e,\eta^\prime,\d,\d/2},\,t \geq \tau_{1}^{\e,\eta^\prime,\d,\d/4}\r)}\\
\vs
\ds{\leq \mathbb{P}_y\le(|H(X_\e(\tau_{1}^{\e,\eta^\prime,\d,\d/4}))-H(X_\e(t-\e^\a))|\geq \d/4\,,\,t -\e^\a<\tau_{1}^{\e,\eta^\prime,\d,\d/2},\,t \geq \tau_{1}^{\e,\eta^\prime,\d,\d/4}\r) }\\
\vs
\ds{\leq \mathbb{P}_y\le(|H(X_\e(t\wedge \tau_{1}^{\e,\eta^\prime,\d,\d/4} ))-H(X_\e((t-\e^\a)\wedge \tau_{1}^{\e,\eta^\prime,\d,\d/4} ))|\geq \d/4\r)}\\
\vs
\ds{\leq \frac{16}{\d^2}\,\mathbb{E}_y\le(|H(X_\e(t\wedge \tau_{1}^{\e,\eta^\prime,\d,\d/4} ))-H(X_\e((t-\e^\a)\wedge \tau_{1}^{\e,\eta^\prime,\d,\d/4} ))|^2\r),}
\end{array}
\end{equation}
and, thanks to \eqref{f56}, this yields
\begin{equation}
\label{f50bis}
\sup_{y \in\,D(\pm \d)}\,|L_3^{\e,\d}(t,y)|\leq \frac{16}{\d^2}\,c_{T,\eta^\prime}\,\e^\a.
\end{equation}

Finally, we consider $L^{\e,\d}_1(t,y)$.  As a consequence of the Markov property, 
\[\begin{array}{l}
\ds{\mathbb{E}_{y}\le(u(X_\e(t))-(u^\wedge)^\vee(X_\e(t-\e^\a))\,;\,t -\e^\a<\tau_{1}^{\e,\eta^\prime,\d,\d/2}\r)}\\
\vs
\ds{=
\mathbb{E}_{y}\le(\psi_\e(\e^\alpha,X_\e(t-\e^\a))\,;\,t -\e^\a<\tau_{1}^{\e,\eta^\prime,\d,\d/2}\r),}
\end{array}\]
where
\[\psi_\e(s,x)=\mathbb{E}_x u(X_\e(s))-(u^\wedge)^\vee(x).\]
Since the family $\{\Pi(X_\e)\}_{\e>0}$ is weakly convergent in $C([0,+\infty);\Gamma)$ and $H(x)\uparrow \infty$, as $|x|\uparrow \infty$, we have that for any $\eta>0$ there exists $M_\eta>0$ such that
\[\begin{array}{l}
\ds{\le|\mathbb{E}_{y}\le(\psi_\e(\e^\alpha,X_\e(t-\e^\a))\,;\,t -\e^\a<\tau_{1}^{\e,\eta^\prime,\d,\d/2}\,,\,|X_\e(t-\e^\a)|> M_\eta\r)\r|}\\
\vs
\ds{\leq 2\,\|u\|_\infty \mathbb{P}_y\le(|X_\e(t-\e^\a)|> M_\eta\r)\leq \eta.}
\end{array}\]
Therefore, 
\[\begin{array}{l}
\ds{|L_1^{\e,\d}(t,y)|\leq \eta+\mathbb{E}_{y}\le(|\psi_\e(\e^\alpha,X_\e(t-\e^\a))|\,;\,t -\e^\a<\tau_{1}^{\e,\eta^\prime,\d,\d/2}\, ,\,|X_\e(t-\e^\a)|\leq M_\eta\r).}
\end{array}\]
As above, we write
\[\begin{array}{l}
\ds{\mathbb{E}_{y}\le(|\psi_\e(\e^\alpha,X_\e(t-\e^\a))|\,;\,t -\e^\a<\tau_{1}^{\e,\eta^\prime,\d,\d/2}\, ,\,|X_\e(t-\e^\a)|\leq M_\eta\r)}\\
\vs
\ds{=\mathbb{E}_{y}\le(|\psi_\e(\e^\alpha,X_\e(t-\e^\a))|\,;\,t -\e^\a<\tau_{1}^{\e,\eta^\prime,\d,\d/2}\, ,\,t\geq \tau_{1}^{\e,\eta^\prime,\d,\d/4}\,\,|X_\e(t-\e^\a)|\leq M_\eta\r)}\\
\vs
\ds{+\mathbb{E}_{y}\le(|\psi_\e(\e^\alpha,X_\e(t-\e^\a))|\,;\,t -\e^\a<\tau_{1}^{\e,\eta^\prime,\d,\d/2}\, ,\,t<\tau_{1}^{\e,\eta^\prime,\d,\d/4}\,\,|X_\e(t-\e^\a)|\leq M_\eta\r)}\\
\vs
\ds{=:I_{\e,1}(t)+I_{\e,2}(t).}\end{array}\]
Due to \eqref{f56} and \eqref{f400}, we have
\[I_{\e,1}(t)\leq 2\,\|u\|_\infty \,c_T\frac{\e^\a}{\d^2}.\]

Now, in \cite[Lemma 4.3]{f02}, it has been proved
that,  as a consequence of the averaging principle, under  the crucial assumption  \eqref{fluid18} given in Hypothesis \ref{H3},  if $\a \in\,(4/7,2/3)$, then for every fixed $\d>0$
\begin{equation}
\label{fluid20}
\lim_{\e\to 0} \sup_{x \in\,K}\,|\psi_\e(\e^\a,x)|=\lim_{\e\to 0} \sup_{x \in\,K}\, \le|\mathbb{E}_x\, u(X_\e(\e^\a))-(u^\wedge)^\vee (x)\r|=0,
\end{equation}
for any compact subset $K$ in  $G(\pm \d/2)^c$ and   any function $u$ whose support is contained in $G(\pm \d/4)^c$.
Since  $X_\e(t) \in\,G(\pm \d/4)^c$, if  $t<\tau_{1}^{\e,\eta^\prime,\d,\d/4}$, \eqref{fluid20} implies that 
\[\lim_{\e\to 0}\sup_{y \in\,D(\pm \d)}\mathbb{E}_{y}\le(|\psi_\e(\e^\alpha,X_\e(t-\e^\a))|\,;\,t -\e^\a<\tau_{1}^{\e,\eta^\prime,\d,\d/2}\, ,\,t<\tau_{1}^{\e,\eta^\prime,\d,\d/4},\,|X_\e(t-\e^\a)|\leq M_\eta\r)=0,\]
and, because of the arbitrariness of $\eta>0$, we conclude that
\[
\lim_{\e\to 0}\sup_{t \in\,[\tau,T]}\,\sup_{y \in\,D(\pm \d)}\,|L_1^{\e,\d}(t,y)|=0.
\]

This, together with   \eqref{f55} and \eqref{f50bis}, implies that for every $\d\leq \d_2$ fixed
\[\lim_{\e\to 0} \sup_{t \in\,[\tau,T]}\,\sum_{n=1}^\infty |J_{2,n}^{\e,\eta^\prime,\d}(t)|=0.\]
 Therefore, if in \eqref{f44} we pick $\bar{\d} \in\,(0,\d_2]$  such that 
 \[\sup_{t \in\,[\tau, T]}\sum_{n \in\,\mathbb{N}}\,J_{1,n}^{\e,\eta^\prime,\bar{\d}}(t)<\frac \eta 4,\]
 and then we pick $\e_\eta>0$ such that 
 \[\sup_{t \in\,[\tau,T]}\,\sum_{n=1}^\infty |J_{2,n}^{\e,\eta^\prime,\bar{\d}}(t)|<\frac \eta 4,\ \ \ \ \e\leq \e_\eta,\]
 because of \eqref{f61} and \eqref{f60}, we can conclude that 
 \[\sup_{t \in\,[\tau,T]}\le|\mathbb{E}_x \le(u(X_\e(t))-(u^\wedge)^\vee(X_\e(t))\,;\,t-\e^\a<\rho_{\e,\eta^\prime}\r)\r|<\frac \eta 2,\ \ \ \ \ \e<\e_\eta,\]
 and \eqref{f40} follows.

 \section{Some consequences of Theorem \ref{fluid-main}}
 \label{sec6}

 The first immediate consequence of Theorem \ref{fluid-main} is that the semigroup $S_\e(t)$ converges to the semigroup $\bar{S}(t)$ in  $H_\gamma$, as $\e\downarrow 0$.
 
 \begin{Corollary}
 Under Hypotheses \ref{H1}, \ref{H2} and \ref{H3}, for every $0<\tau<T$ and $u \in\,H_\gamma$ we have
 \begin{equation}
 \label{f70}
 \lim_{\e\to 0} \sup_{t \in\,[\tau,T]}|S_\e(t) u-\bar{S}(t)^{\vee} u|_{H_\gamma}= \lim_{\e\to 0} \sup_{t \in\,[\tau,T]}|(S_\e(t) u)^\wedge -\bar{S}(t) u^\wedge|_{\bar{H}_\gamma}=0.
 \end{equation}
  \end{Corollary}
  
  \begin{proof}

First of all, we notice that in view of \eqref{fluid7}, the first limit in \eqref{f70} implies the second one. So, we will only prove the first limit. We have
\[|S_\e(t) u-\bar{S}(t)^{\vee} u|_{H_\gamma}^2=\int_{\mathbb{R}^2}|S_\e(t) u(x)-\bar{S}(t)^{\vee} u(x)|^2\gamma^\vee(x)\,dx.\]
If $u \in\,C_b(\mathbb{R}^2)$, we have
\[\sup_{t\geq 0}|S_\e(t) u(x)-\bar{S}(t)^{\vee} u(x)|\leq 2\,|u|_\infty.\]
Hence, since $\gamma^\vee \in\,L^1(\mathbb{R}^2)$, in view of \eqref{fluid17} and the dominated convergence theorem, we have that the first limit in \eqref{f70} is true for every $u \in\,C_b(\mathbb{R}^2)$. Moreover, in view of Hypothesis \ref{H2}, we have that 
\begin{equation}
\label{f73}
|S_\e(t) u|_{H_\gamma}\leq c_T\,|u|_{H_\gamma},\ \ \ \ \ t \in\,[0,T],\ \ \ \e>0,\end{equation}
so that $\bar{S}(t)^\vee u \in\,H_\gamma$ and
\begin{equation}
\label{f72}
|\bar{S}(t) ^\vee u|_{H_\gamma}\leq c_T\,|u|_{H_\gamma},\ \ \ \ \ t \in\,[0,T].\end{equation}

Since $C_c(\mathbb{R}^2)$ is dense in $L^2(\mathbb{R}^2)$ and we assume the weight $\gamma^\vee$ to be continuous and strictly positive, we have that $C_c(\mathbb{R}^2)$ is dense in $H_\gamma$. Actually, if $u \in\,H_\gamma$, then $u\sqrt{\gamma^\vee} \in\,L^2(\mathbb{R}^2)$. Then, if $\{u_n\}_{n \in\,\nat}$ is a sequence in $C_c(\mathbb{R}^2)$, converging to $u\sqrt{\gamma^\vee} \in\,H_\gamma$ in $L^2(\mathbb{R}^2)$, we have that the sequence $\{u_n/\sqrt{\gamma^\vee}\}_{n \in\,\nat}$ converges to $u$ in $H_\gamma$. Moreover, as $u_n$ has compact support and $1/\sqrt{\gamma^\vee}$ is continuous and positive, it follows that $\{u_n/\sqrt{\gamma^\vee}\}_{n \in\,\nat}\subset C_c(\mathbb{R}^2)$.

Thus, for any  $u \in\,H_\gamma$, we can fix a sequence $\{u_n\}_{n \in\,\nat }\subset C_c(\mathbb{R}^2)$ converging to $u$ in $H_\gamma$. Thanks to \eqref{f72}, we get that the sequence $\{\bar{S}(t)^\vee u_n\}_{n \in\,\nat}$ is Cauchy in $H_\gamma$, so that we conclude that
\[\exists \lim_{n\to \infty}\bar{S}(t)^\vee u_n=:\bar{S}(t)^\vee u \in\,H_\gamma,\]
the limit does not depend on the sequence $\{u_n\}_{n \in\,\nat}$ and 
and \eqref{f72} holds for every $u \in\,H_\gamma$.

Finally, since we have
\[\begin{array}{l}
\ds{|S_\e(t)u-\bar{S}(t)^\vee u|_{H_\gamma}\leq |S_\e(t)(u-u_n)|_{H_\gamma}+|\bar{S}(t)^\vee (u-u_n)|_{H_\gamma}+|S_\e(t)u_n-\bar{S}(t)^\vee u_n|_{H_\gamma},}
\end{array}\]
according to \eqref{f73} and \eqref{f72}, for every $\eta>0$ we can find $\eta_\eta \in\,\nat$ such that 
\[\sup_{t \in\,[\tau,T]}|S_\e(t)u-\bar{S}(t)^\vee u|_{H_\gamma}\leq \eta +\sup_{t \in\,[\tau,T]}|S_\e(t)u_{n_\eta}-\bar{S}(t)^\vee u_{n_\eta}|_{H_\gamma}.\]  This allows to conclude, as $u_{u_\eta} \in\,C_c(\mathbb{R}^2)$.
  \end{proof}

\begin{Remark}
{\em    From the proof of the corollary above, it is  clear that from the pointwise convergence of $S_\e(t)u$  to $\bar{S}(t)^\vee u$, as stated in Theorem \ref{fluid-main}, we cannot conclude that limit \eqref{f70} is also true in $L^2(\mathbb{R}^2)$, as the Lebesgue measure in $\mathbb{R}^2$ is not finite. It is only after introducing a weight that we can prove the convergence in ${H}_\gamma$. }
\end{Remark}

\begin{Corollary}
Under Hypotheses \ref{H1}, \ref{H2} and \ref{H3}, we have that the semigroup $\bar{S}(t)$ is well defined in $\bar{H}_\gamma$ and for any $T>0$ there exists $c_T>0$ such that
\begin{equation}
\label{f87}
\|\bar{S}(t)\|_{\mathcal{L}(\bar{H}_\gamma)}\leq c_T,\ \ \ \ t \in\,[0,T].\end{equation}
\end{Corollary}

\begin{proof}
In \eqref{f72} we have seen that $\bar{S}(t)^\vee$ is well defined in $H_\gamma$ and for any $T>0$ there exists $c_T>0$ such that for any $u \in\,H_\gamma$
\[|\bar{S}(t) ^\vee u|_{H_\gamma}\leq c_T\,|u|_{H_\gamma},\ \ \ \ \ t \in\,[0,T].\]
Therefore, thanks to \eqref{fluid8} and \eqref{fluid7}, if $f \in\,\bar{H}$
\[|\bar{S}(t)f|_{\bar{H}_\gamma}=|(\bar{S}(t)f)^\vee|_{{H}_\gamma}=|\bar{S}(t)^\vee (f^\vee)|_{{H}_\gamma}
\leq c_T\,|f^\vee|_{{H}_\gamma}=c_T\,|f|_{\bar{H}_\gamma},\ \ \ \ t \in\,[0,T],\]
and this allows to conclude.

\end{proof}

\section{From the SPDE on $\mathbb{R}^2$ to the SPDE on the graph}
\label{sec5}

We are interested here in the equation
\begin{equation}
\label{stoch-eq}
\le\{\begin{array}{l}
\ds{\partial_t u_\e(t,x)=\mathcal{L}_\e u_\e(t,x)+b(u_\e(t,x))+g(u_\e(t,x))\partial_t \mathcal{W}(t,x),\ \ \ \ t>0,}\\
\vs
\ds{u_\e(0,x)=\varphi(x),\ \ \ x \in\,\mathbb{R}^2.}
\end{array}\r.
\end{equation}
 $\mathcal{L}_\e$ is the second order differential operator defined by
\[\mathcal{L}_\e u(x)=\frac 12 \,\Delta u(x)+\frac 1\e \langle \bar{\nabla }H(x),\nabla u(x)\rangle,\ \ \ \ x \in\,\mathbb{R}^2,\]
associated with equation \eqref{ham-epsilon-bis} and with Markov transition semigroup $S_\e(t)$.
The Hamiltonian $H$ satisfies Hypotheses \ref{H1} and \ref{H2} and and the nonlinearities $b, g:\mathbb{R}^2\to \mathbb{R}$ are assumed to be Lipschitz continuous. 

Concerning the random forcing $\mathcal{W}$, we assume that it is a {\em spatially homogeneous Wiener process}, with finite spectral measure $\mu$ (see e.g. \cite{pz} and \cite{DPZ} for all details). This means that there exists a Gaussian random field on $[0,+\infty)\times \mathbb{R}^2$,  that we also denote by $\mathcal{W}$,  defined on some stochastic basis $(\boldsymbol{\Omega}, \boldsymbol{\mathbf{\mathcal{F}}}, \{\boldsymbol{\mathcal{F}}_t\}_{t\geq 0},\mathbf{{P}})$, such that
\begin{enumerate}
\item the mapping $(t,x)\mapsto \mathcal{W}(t,x)$ is continuous with respect to $t$ and measurable with respect to both variables, $\mathbf{P}$-almost surely;
\item for each $x \in\,\mathbb{R}^2$, the process $\mathcal{W}(t,x)$, $t\geq 0$, is a one-dimensional Wiener process;
\item for every $t,s \geq 0$ and $x, y \in\,\mathbb{R}^2$
\begin{equation}
\label{f83}
\mathbf{E}\, \mathcal{W}(t,x)\mathcal{W}(s,y)=(t\wedge s)\, \Lambda(x-y),
\end{equation}
where $\Lambda$ is the  Fourier transform of the spectral measure $\mu$, that is
\[\Lambda(x)=\int_{\mathbb{R}^2} e^{i\langle x,\la\rangle}\,\mu(d\la),\ \ \ \ x \in\,\mathbb{R}^2.\]
\end{enumerate}

Notice that, with this definition, $\mathcal{W}(t,\cdot) \in\,L^2(\boldsymbol{\Omega};H_\gamma)$. Actually, we have
\[\mathbf{E}|\mathcal{W}(t,\cdot)|_{H_\gamma}^2=\mathbf{E}\int_{\mathbb{R}^2}|\mathcal{W}(t,x)|^2\gamma^\vee(x)\,dx=\int_{\mathbb{R}^2}\mathbf{E}|\mathcal{W}(t,x)|^2\gamma^\vee(x)\,dx=t\,\Lambda(0)\int_{\mathbb{R}^2}\gamma^\vee(x)\,dx.\]

In what follows, we  denote by $L^2_{(s)}(\mathbb{R}^2,d\mu)$  the subspace of the Hilbert space $L^2(\mathbb{R}^2,d\mu;\mathbb{C})$ consisting of all functions $\varphi$ such that $\varphi_{(s)}=\varphi$, where
\[\varphi_{(s)}(x)=\overline{\varphi(-x)},\ \ \ x \in\,\mathbb{R}^2.\]
Moreover, we denote by $RK$ the  {\em reproducing kernel} Hilbert space of the Wiener process $\mathcal{W}$ (see \cite{DPZ} for the definition).

As shown in \cite[Proposition 1.2]{pz}, an orthonormal basis for the reproducing kernel $RK$ is given by $\{\widehat{u_j \mu}\}_{j \in\,\nat}$, where $\{u_j\}_{j \in\,\nat}$ is a complete orthonormal basis of  the Hilbert space $L^2_{(s)}(\mathbb{R}^2,d\mu)$ and 
\[\widehat{u_j \mu}(x)=\int_{\mathbb{R}^2} e^{i\langle x,\la\rangle}u_j(\la)\,\mu(d\la),\ \ \ \ x \in\,\mathbb{R}^2.\]
This means, in particular, that $\mathcal{W}(t,x)$ can be represented as 
\begin{equation}
\label{f81}
\mathcal{W}(t,x)=\sum_{j=1}^\infty\, \widehat{u_j \mu}(x)\,\beta_j(t),\ \ \ \ t\geq 0,\end{equation}
for some  sequence of independent Brownian motions $\{\beta_j\}_{j \in\,\nat}$, all defined on the same stochastic basis.
Moreover, in \cite{pz} it is also shown that 
\begin{equation}
\label{f86}
\sum_{j=1}^\infty |\widehat{u_j \mu}|^2_{H_\gamma}<\infty.
\end{equation}

For every $u \in\,H_\gamma$ and $v$ in the reproducing kernel of $\mathcal{W}$, we shall denote
\[B(u)(x)=b(u(x)),\ \ \ \ G(u)v(x)=[g(u)v](x),\ \ \ \ x \in\,\mathbb{R}^2.\]
Since we are assuming $b$ to be Lipschitz continuous, we have that
$B:H_\gamma\to H_\gamma$ is Lipschitz continuous. Moreover, as proved in \cite[Lemma 4.1]{pz}, $G$ maps   $H_\gamma$ into $\mathcal{L}_2(RK,H_\gamma)$, where $\mathcal{L}_2(RK,H_\gamma)$ is the space of Hilbert-Schmidt operators defined on $RK$ with values in $H_\gamma$. Furthermore, for every $u_1, u_2 \in\,H_\gamma$
\begin{equation}
\label{f107}
|G(u_1)-G(u_2)|_{\mathcal{L}_2(RK,H_\gamma)}\leq c\,|u_1-u_2|_{H_\gamma}.\end{equation}
By using a stochastic factorization argument and \eqref{f65}, this implies that 
\begin{equation}
\label{f91}
\mathbf{E}\sup_{t \in\,[0,T]} \le|\int_0^t S_\e(t-s)\,\le[G(u_1(s))-G(u_2(s))\r]\,d\mathcal{W}(s)\r|_{H_\gamma}^p\leq c_p(T)\,\mathbf{E}\sup_{t \in\,[0,T]} |u_1(t)-u_2(t)|^p_{H_\gamma}.
\end{equation}

\begin{Definition}
A predictable process $u_\e(t)$, taking values in $H_\gamma$ is a mild solution to equation \eqref{stoch-eq} if it satisfies the following integral equation
\begin{equation}
\label{mild}
u_\e(t)=S_\e(t)\varphi+\int_0^t S_\e(t-s) B(u_\e(s))\,ds+\int_0^t S_\e(t-s)\,G(u_\e(s))\,d\mathcal{W}(s).\end{equation}
\end{Definition}

In \cite[Theorem 2.1]{pz} it is proved that, under our assumptions on the Hamiltonian $H$, the coefficients $b$ and $g$ and the noise $\mathcal{W}$, for every $\varphi \in\,H_\gamma$ and every $p\geq 1$ and $T>0$ there exists a unique mild solution $u_\e$ for equation \eqref{stoch-eq} in $L^p(\boldsymbol{\Omega};C([0,T];H_\gamma))$. Moreover,  we have
\begin{equation}
\label{f80}
\mathbb{E}\sup_{t \in\,[0,T]}|u_\e(t)|_{H_\gamma}^p=c_p(T),\ \ \ \ \e>0.
\end{equation}

Our purpose  is studying the limiting behavior of $u_\e$ in the space $L^p(\boldsymbol{\Omega};C([0,T];H_\gamma))$ or, equivalently, the limit of $u_\e^\wedge $ in the space $L^p(\boldsymbol{\Omega};C([0,T];\bar{H}_\gamma))$, as $\e\downarrow 0$. 

The limiting process will be the solution  $\bar{u}$ of the following SPDE on the graph $\Gamma$
\begin{equation}
\label{eq-graph}
\le\{\begin{array}{l}
\ds{\partial_t \bar{u}(t,z,k)=\bar{{L}}\,\bar{u}(t,z,k)+b(\bar{u}(t,z,k))+g(\bar{u}(t,z,k))\,\partial_t \bar{\mathcal{W}}(t,z,k),\ \ \ \ t>0,}\\
\vs
\ds{\bar{u}(0,z,k)=\varphi^\wedge(z,k),\ \ \ (z,k) \in\,\Gamma.}
\end{array}\r.
\end{equation}
Here $\bar{L}$ is the differential operator  on the graph $\Gamma$,  introduced in Subsection \ref{ssec2.3}, generator of the limiting Markov process $\bar{Y}(t)$. Concerning the noisy forcing $\bar{\mathcal{W}}$, it is defined by  
\[\bar{\mathcal{W}}(t,z,k)=\sum_{j=1}^\infty\, (\widehat{u_j \mu})^\wedge (z,k)\,\beta_j(t),\ \ \ \ t\geq 0\ \ \ (z,k) \in\,\Gamma,\]
where $\{\beta_j\}_{j \in\,\nat}$ is the sequence of independent Brownian motions introduced in \eqref{f81}.
We have
\[\begin{array}{l}
\ds{\mathbf{E}\,\bar{\mathcal{W}}(t,z_1,k_1)\bar{\mathcal{W}}(s,z_2,k_2)=  \mathbf{E}\,\oint_{C_{k_1}(z_1)}\mathcal{W}(t,x)\,d\mu_{z_1,k_1}\oint_{C_{k_2}(z_2)}\mathcal{W}(t,y)\,d\mu_{z_2,k_2} }\\
\vs
\ds{=\oint_{C_{k_1}(z_1)}\oint_{C_{k_2}(z_2)}\mathbf{E}\,\mathcal{W}(t,x)\mathcal{W}(t,y)\,\,d\mu_{z_1,k_1}\,d\mu_{z_2,k_2}.}
\end{array}\]
Thanks to \eqref{f83},
this gives
\begin{equation}
\label{f84}
\begin{array}{l}
\ds{\mathbf{E}\,\bar{\mathcal{W}}(t,z_1,k_1)\bar{\mathcal{W}}(s,z_2,k_2)}\\
\vs
\ds{=(t\wedge s)\oint_{C_{k_1}(z_1)}\oint_{C_{k_2}(z_2)}\int_{\mathbb{R}^2} e^{i\langle \la,x-y\rangle}\,\mu(d\la)\,\,d\mu_{z_1,k_1}\,d\mu_{z_2,k_2}}\\
\vs
\ds{=(t\wedge s)\int_{\mathbb{R}^2}\le[\oint_{C_{k_1}(z_1)}e^{i\langle \la,x\rangle}\,d\mu_{z_1,k_1}\ \oint_{C_{k_2}(z_2)}e^{-i\langle \la,y\rangle}\,d\mu_{z_2,k_2}\r]\,\mu(d\la)}\\
\vs
\ds{=(t\wedge s)\int_{\mathbb{R}^2} \le(e^{i\langle \la,\cdot\rangle}\r)^{\wedge} (z_1,k_1) \le(e^{-i\langle \la,\cdot\rangle}\r)^{\wedge} (z_2,k_2)\,\mu(d\la).}
\end{array}\end{equation}
It is immediate to check that $(\bar{u})^\wedge=\overline{u^\wedge}$, for every $u:\mathbb{R}^2\to \mathbb{C}$.
Hence, thanks to \eqref{f84} and \eqref{fluid7}, this allows to conclude
\[\begin{array}{l}
\ds{\mathbf{E}\,|\bar{\mathcal{W}}(t)|^2_{\bar{H}_\gamma}=\sum_{k=1}^n\int_{I_k} \mathbf{E}\,|\bar{\mathcal{W}}(t,z,k)|^2 T_k(z)\gamma(z,k)\,dz}\\
\vs
\ds{=t\sum_{k=1}^n\int_{I_k}\int_{\mathbb{R}^2}\le(e^{i\langle \la,\cdot\rangle}\r)^{\wedge} \le(e^{-i\langle \la,\cdot\rangle}\r)^{\wedge} (z,k)\,\mu(d\la)\ T_k(z)\,\gamma(z,k)\,dz}\\
\vs
\ds{=t\int_{\mathbb{R}^2}\le[\,\sum_{k=1}^n\int_{I_k}\big|\le( e^{i\langle \la,\cdot\rangle}\r)^\wedge\big|^2 (z,k)\ T_k(z)\,\gamma(z,k)\,dz\r]\mu(d\la)=t\,\int_{\mathbb{R}^2}\big|\le(e^{i\langle \la,\cdot\rangle}\r)^\wedge\big|^2_{\bar{H}_\gamma}\,\mu(d\la)}\\
\vs
\ds{\leq t\,\int_{\mathbb{R}^2}\big|e^{i\langle \la,\cdot\rangle}\big|^2_{{H}_\gamma}\,\mu(d\la)=t\,\mu(\mathbb{R}^2)\int_{\mathbb{R}^2}\gamma^\vee(x)\,dx<\infty.}
\end{array}\]
This means that $\bar{\mathcal{W}}(t) \in\,L^2(\boldsymbol{\Omega};\bar{H}_\gamma)$, for every $t\geq 0$.

Moreover, due to \eqref{fluid7} and \eqref{f86}, we have
\[\sum_{j=1}^\infty |(\widehat{u_j \mu})^\wedge|^2_{\bar{H}_\gamma}\leq \sum_{j=1}^\infty |\widehat{u_j \mu}|^2_{{H}_\gamma}<\infty.\]
Then, in view of  \eqref{f87}, we have that for every $v \in\,L^p(\boldsymbol{\Omega};C([0,T];\bar{H}_\gamma))$ the stochastic convolution
\[ t \in\,[0,+\infty)\mapsto \int_0^t \bar{S}(t-s)\,G(v(s))\,d\bar{\mathcal{W}}(s)=
\sum_{k=1}^\infty\int_0^t \bar{S}(t-s) \le[G(v(s))(\widehat{u_j \mu})^\wedge\r]\,d\beta_j(s),\] is well defined in $L^2(\boldsymbol{\Omega};\bar{H}_\gamma)$ and if $v_1, v_2 \in\,L^p(\boldsymbol{\Omega};C([0,T];\bar{H}_\gamma))$, we have
\begin{equation}
\label{f89}
\mathbf{E}\sup_{t \in\,[0,T]}\,\le|\int_0^t \bar{S}(t-s)\,\le[G(v_1(s))-G(v_2(s))\r]\,d\bar{\mathcal{W}}(s)\r|^2_{\bar{H}_\gamma}
\leq c(T)\,\mathbf{E}\sup_{t \in\,[0,T]}\,|v_1(t)-v_2(t)|^2_{\bar{H}_\gamma}.
\end{equation}
\begin{Definition}
A predictable process $\bar{u}(t)$, taking values in $\bar{H}_\gamma$, is a mild solution to equation \eqref{eq-graph} if it satisfies the following integral equation
\begin{equation}
\label{mild}
\bar{u}(t)=\bar{S}(t)\varphi ^\wedge +\int_0^t \bar{S}(t-s) B(\bar{u}(s))\,ds+\int_0^t \bar{S}(t-s)\,G(\bar{u}(s))\,d\bar{\mathcal{W}}(s).\end{equation}
\end{Definition}
Since  the mapping $B$ can be extended to $\bar{H}_\gamma$, as a Lipschitz continuous mapping, due to \eqref{f89} we can conclude that there exists a unique mild  solution $\bar{u}$ to equation \eqref{eq-graph} in $L^p(\boldsymbol{\Omega},C([0,T];\bar{H}_\gamma))$, for any $T>0$ and $p\geq 1$.  Moreover
\begin{equation}
\label{f96}
\mathbf{E}\sup_{t \in\,[0,T]}|\bar{u}(t)|_{\bar{H}_\gamma}^p<+\infty.
\end{equation}

\begin{Theorem}
Assume Hypotheses \ref{H1}, \ref{H2} and \ref{H3}. Then, for any $\varphi \in\,H_\gamma$, $p\geq 1$ and $0<\tau<T$  we have
\begin{equation}
\label{f90}
\lim_{\e\to 0}\mathbf{E}\sup_{ t \in\,[\tau,T]}|u_\e(t)-\bar{u}(t)^\vee|^p_{H_\gamma}=\lim_{\e\to 0}\mathbf{E}\sup_{ t \in\,[\tau,T]}|u_\e (t)^\wedge-\bar{u}(t)|^p_{\bar{H}_\gamma}=0.
\end{equation}
\end{Theorem}

\begin{proof}
As $u_\e$ and $\bar{u}$ are mild solutions to equation \eqref{stoch-eq} and \eqref{eq-graph}, respectively, we have
\[\begin{array}{l}
\ds{u_\e(t)-\bar{u}(t)^\vee=\le[S_\e(t) \varphi-\bar{S}(t)^\vee \varphi\r]}\\
\vs
\ds{+\le[\int_0^t S_\e(t-s) B(u_\e(s))\,ds-\int_0^t \bar{S}(t-s) ^\vee B(\bar{u}(s)^\vee)\,ds\r]+\le[\Theta_\e(t)-\bar{\Theta}^\vee(t)\r]=:\sum_{i=1}^3 I_{\e,i}(t),}
\end{array}\]
where, for the sake of brevity, we have defined
\[\Theta_\e(t):=\int_0^t S_\e(t-s)\,G(u_\e(s))\,d\mathcal{W}(s),\ \ \ \bar{\Theta}(t):=\int_0^t \bar{S}(t-s)\,G(\bar{u}(s))\,d\bar{\mathcal{W}}(s).\]

In order to conclude the proof of \eqref{f90}, we need a couple of Lemmas, whose proof is postponed to the end of this section.

\begin{Lemma}
\label{l99}
For every $p\geq 1$, $T>0$ and $0<\tau\leq t\leq T$
\begin{equation}
\label{f93}
\mathbf{E}\sup_{ s \in\,[0,t]}
|I_{\e,2}(s)|_{H_\gamma}^p\leq c_p(T)\int_\tau^t\mathbf{E}\sup_{r \in\,[\tau,s]}|u_\e(r)-\bar{u}(r)^\vee|_{H_\gamma}^p\,ds+R_{T,p}(\tau,\e),
\end{equation}
for some constant $R_{T,p}(\tau,\e)$ such that
\begin{equation}
\label{f103}
\lim_{\e, \tau\to 0}R_{T,p}(\tau,\e)=0.\end{equation}
\end{Lemma}

\begin{Lemma}
\label{l100}
For every $p\geq 1$, $T>0$ and $0<\tau\leq t\leq T$
\begin{equation}
\label{f933}
\mathbf{E}\sup_{ s \in\,[0,t]}
|I_{\e,3}(s)|_{H_\gamma}^p\leq c_p(T)\int_\tau^t\mathbf{E}\sup_{r \in\,[\tau,s]}|u_\e(r)-\bar{u}(r)^\vee|_{H_\gamma}^p\,ds+S_{T,p}(\tau,\e),
\end{equation}
for some constant $S_{T,p}(\tau,\e)$ such that
\begin{equation}
\label{f1033}
\lim_{\e, \tau\to 0}S_{T,p}(\tau,\e)=0.\end{equation}

\end{Lemma}

Now, because of Lemmas \ref{l99} and \ref{l100}, we have
\[\begin{array}{l}
\ds{\mathbf{E}\sup_{s \in\,[\tau,t]}|u_\e(s)-\bar{u}^\vee(s)|_{H_\gamma}^p\leq c\le(\sup_{s \in\,[\tau,T]} |S_\e(s)\varphi-\bar{S}(s)^\vee \varphi|^p_{H_\gamma}+R_{T,p}(\tau,\e)+S_{T,p}(\tau,\e)\r)}\\
\vs
\ds{+   c_p(T) \int_\tau^t \mathbf{E}\sup_{r \in\,[\tau,s]}|u_\e(r)-\bar{u}^\vee(r)|_{H_\gamma}^p\,ds.}
\end{array}\]
Then, thanks the Gronwall lemma
\[\begin{array}{l}
\ds{\mathbf{E}\sup_{s \in\,[\tau,T]}|u_\e(s)-\bar{u}^\vee(s)|_{H_\gamma}^p}\\
\vs
\ds{\leq c_p\,e^{c_p(T)(T-\tau)}\le[\sup_{s \in\,[\tau,T]} |S_\e(s)\varphi-\bar{S}(s)^\vee \varphi|^p_{H_\gamma}+R_{T,p}(\tau,\e)+S_{T,p}(\tau,\e)\r].}
\end{array}\]
and we can conclude, thanks to \eqref{f70}, \eqref{f103} and \eqref{f1033}.

\end{proof}

\section{Proof of Lemmas \ref{l99} and \ref{l100}}

\begin{proof}[Proof of Lemma \ref{l99}.]
We have
\begin{equation}
\label{f92}
\begin{array}{l}
\ds{I_{\e,2}(t)=\int_0^t S_\e(t-s) \le[B(u_\e(s))-B(\bar{u}(s)^\vee)\r]\,ds}\\
\vs
\ds{+\int_0^t  \le[S_\e(t-s)-\bar{S}(t-s)^\vee\r] B(\bar{u} (s)^\vee)\,ds=:J_{\e,1}(t)+J_{\e,2}(t).}
\end{array}
\end{equation}
If we fix $0<\tau \leq t\leq T$, , we have
\[\begin{array}{l}
\ds{|J_{\e,1}(t)|_{H_\gamma}^p\leq c_p(T)\le(\int_0^\tau \le(|u_\e(s)|_{H_\gamma}^p+|\bar{u}(s)^\vee|_{H_\gamma}^p+1\r)\,ds+\int_\tau ^t |u_\e(s)-\bar{u}(s)^\vee|_{H_\gamma}^p\,ds\r)}\\
\vs
\ds{\leq c_p(T)\sup_{s \in\,[0,T]} \le(|u_\e(s)|_{H_\gamma}^p+|\bar{u}(s)^\vee|_{H_\gamma}^p+1\r)\tau+c_p(T)\int_\tau^t \sup_{r \in\,[\tau,s]}|u_\e(r)-\bar{u}(r)^\vee|_{H_\gamma}^p\,ds.}
\end{array}\]
Hence,  from  \eqref{f80} and \eqref{f96} we can conclude
\begin{equation}
\label{f100}
\mathbf{E}\sup_{s \in\,[0,t]}|J_{\e,1}(s)|_{H_\gamma}^p\leq c_p(T)\,\tau+c_p(T)\int_\tau^t \mathbf{E}\,\sup_{r \in\,[\tau,s]}|u_\e(r)-\bar{u}(r)^\vee|_{H_\gamma}^p\,ds.
\end{equation}
Concerning $J_{\e,2}(t)$, we have
\[\begin{array}{l}
\ds{|J_{\e,2}(t)|_{H_\gamma}^p}\\
\vs
\ds{\leq c_p(T)\int_0^{t-\tau}\le|\le[S_\e(t-s)-\bar{S}(t-s)^\vee\r] B(\bar{u} (s)^\vee)\r|_{H_\gamma}^p\,ds+c_p(T)\int_{t-\tau}^t \le(1+|\bar{u} (s)^\vee|^p_{H_\gamma}\r)\,ds}\\
\vs
\ds{\leq c_p(T)\int_0^T \sup_{r \in\,[\tau,T]} \le|\le[S_\e(r)-\bar{S}(r)^\vee\r] B(\bar{u} (s)^\vee)\r|_{H_\gamma}^p\,ds+c_p(T) \le(1+\sup_{s \in\,[0,T]}|\bar{u} (s)^\vee|^p_{H_\gamma}\r)\tau.}
\end{array}\]
Then, according to \eqref{f96}, we conclude
\[\mathbf{E}\sup_{t \in\,[0,T]} |J_{\e,2}(t)|_{H_\gamma}^p\leq c_p(T)\le(\tau+\mathbf{E}\int_0^T \sup_{r \in\,[\tau,T]} \le|\le[S_\e(r)-\bar{S}(r)^\vee\r] B(\bar{u} (s)^\vee)\r|_{H_\gamma}^p\,ds\r).\]
This, together with \eqref{f100}, implies \eqref{f93}, with
\[R_{T,p}(\tau,\e):=c_p(T)\tau+\mathbf{E}\int_0^T \sup_{r \in\,[\tau,T]} \le|\le[S_\e(r)-\bar{S}(r)^\vee\r] B(\bar{u} (s)^\vee)\r|_{H_\gamma}^p\,ds.\]
Moreover, \eqref{f103} follows as a consequence of Theorem \ref{fluid-main} and the dominated convergence theorem, due to the equi-boundedness of the norm of $S_\e(t)$ and $\bar{S}(t)$, and to \eqref{f96}.

\end{proof}

\bigskip

\begin{proof}[Proof of Lemma \ref{l100}]

First of all, we notice that, due to \eqref{fluid12}, 
\[\bar{\Theta}^\vee(t):=\int_0^t \bar{S}(t-s)^\vee\,G(\bar{u}(s)^\vee)\,d{\mathcal{W}}(s).\]
Then
\[\begin{array}{l}
\ds{I_{\e,3}(t)=\int_0^t S_\e(t-s)\le[G(u_\e(s))-G(\bar{u}(s)^\vee)\r]\,d\mathcal{W}(s)}\\
\vs
\ds{+\int_0^t \le[S_\e(t-s)-\bar{S}(t-s)^\vee\r]G(\bar{u}(s)^\vee)\,d\mathcal{W}(s)=:J_{\e,1}(t)+J_{\e,2}(t).}
\end{array}\]
By using a factorization argument, for every $t \in\,[0,T]$ and $\a \in\,(0,1/2)$ we have
\[\begin{array}{l}
\ds{J_{\e,1}(t)= \frac{\sin \pi \a}{\pi}\,\int_0^t(t-s)^{\a-1}S_\e(t-s) Y_\a(s)\,ds,}
\end{array}\]
where
\[Y_\a(s)=\int_0^s (s-\si)^{-\a}S_\e(s-\si)\le[G(u_\e(\si))-G(\bar{u}(\si)^\vee)\r]\,d\mathcal{W}(\si).\] 
Therefore, thanks to \eqref{f65}, for every $p\geq 1/\a$ 
\begin{equation}
\label{f105}
\begin{array}{l}
\ds{\mathbf{E}\sup_{s \in\,[0,t]}|J_{\e,1}(s)|^p_{H_\gamma}\leq c_{p,\a}(T)  \int_0^t\mathbf{E}\,|Y_\a(s)|_{H_\gamma}^p\,ds.  }
\end{array}\end{equation}
In view of  \eqref{f65}  and \eqref{f107}, we have
\[\begin{array}{l}
\ds{\mathbf{E}\,|Y_\a(s)|_{H_\gamma}^p=c_p\,\mathbf{E}\le(\int_0^s (s-\si)^{-2\a}\sum_{j=1}^\infty\le|S_\e(s-\si)\le[G(u_\e(\si))-G(\bar{u}(\si)^\vee)\r]\widehat{u_j\mu}\r|_{H_\gamma}^2\,d\si\r)^{\frac p2}}\\
\vs
\ds{\leq c_p(T)\mathbf{E}\le(\int_0^s (s-\si)^{-2\a}\sum_{j=1}^\infty\le|\le[G(u_\e(\si))-G(\bar{u}(\si)^\vee)\r]\widehat{u_j\mu}\r|_{H_\gamma}^2\,d\si\r)^{\frac p2}}\\
\vs
\ds{=c_p(T)\mathbf{E}\le(\int_0^s (s-\si)^{-2\a}\|G(u_\e(\si))-G(\bar{u}(\si)^\vee)\|_{\mathcal{L}_2(RK,H_\gamma}^2)\,d\si\r)^{\frac p2}}\\
\vs
\ds{\leq c_p(T)\mathbf{E}\le(\int_0^s (s-\si)^{-2\a} |u_\e(\si)-\bar{u}(\si)^\vee|_{H_\gamma}^2\,d\si\r)^{\frac p2}.}
\end{array}\]
Hence, thanks to the Young inequality, due to \eqref{f80} and \eqref{f96}, for any $0<\tau<t<T$
\[\begin{array}{l}
\ds{\int_0^t\mathbf{E}\,|Y_\a(s)|_{H_\gamma}^p\,ds\leq c_p(T)\le(\int_0^t s^{-2\a}\,ds\r)^{\frac p2}\int_0^t\mathbf{E}\,|u_\e(s)-\bar{u}(s)^\vee|_{H_\gamma}^p\,ds}\\
\vs
\ds{\leq c_p(T)\,\tau\,\le( \mathbf{E}\sup_{s \in\,[0,T]}|u_\e(s)|_{H_\gamma}^p+\mathbf{E}\sup_{s \in\,[0,T]}|\bar{u}(s)|_{H_\gamma}^p\r)+c_p(T)\int_\tau^t\mathbf{E}\,\sup_{r \in\,[\tau,s]} |u_\e(r)-\bar{u}(r)^\vee|_{H_\gamma}^p\,ds}\\
\vs
\ds{\leq c_p(T)\le(\tau+\int_\tau^t\mathbf{E}\,\sup_{r \in\,[\tau,s]} |u_\e(r)-\bar{u}(r)^\vee|_{H_\gamma}^p\,ds\r).}
\end{array}\]
According to \eqref{f105}, this yields
\begin{equation}
\label{f109}
\mathbf{E}\sup_{s \in\,[0,t]}|J_{\e,1}(s)|^p_{H_\gamma}\leq c_p(T)\le(\tau+\int_\tau^t\mathbf{E}\,\sup_{r \in\,[\tau,s]} |u_\e(r)-\bar{u}(r)^\vee|_{H_\gamma}^p\,ds\r).
\end{equation}

Next, by using again a factorization argument, for every $t \in\,[0,T]$ and $\a \in\,(0,1/2)$ we have
\[\begin{array}{l}
\ds{J_{\e,2}(t)= \frac{\sin \pi \a}{\pi}\,\int_0^t(t-s)^{\a-1}S_\e(t-s) Y_{\a,1}(s)\,ds}\\
\vs
\ds{+\frac{\sin \pi \a}{\pi}\,\int_0^t(t-s)^{\a-1}\le[S_\e(t-s)-\bar{S}(t-s)^\vee\r] Y_{\a,2}(s)\,ds,}
\end{array}\]
where
\[Y_{\a,1}(s)=\int_0^s (s-\si)^{-\a}\le[S_\e(s-\si)-\bar{S}(s-\si)^\vee\r]G(\bar{u}(\si)^\vee)\,d\mathcal{W}(\si),\]
and
\[ Y_{\a,2}(s)=\int_0^s (s-\si)^{-\a}\bar{S}(s-\si)^\vee\,G(\bar{u}(\si)^\vee)\,d\mathcal{W}(\si).\]
Thus, by proceeding as above for $J_{\e,1}(t)$, thanks to \eqref{f65} we get
\[\begin{array}{l}
\ds{\mathbf{E}\sup_{s \in\,[0,t]}|J_{\e,2}(s)|_{H_\gamma}^p\leq c_{p,\a}(T) \int_0^t \mathbf{E}|Y_{\a,1}(s)|_{H_\gamma}^p\,ds}\\
\vs
\ds{+c_{p,\a}(T)\,\mathbf{E}\sup_{s \in\,[0,t]}\int_0^s \le|\le[S_\e(s-r)-\bar{S}(s-r)^\vee\r] Y_{\a,2}(r)\r|_{H_\gamma}^p\,dr=:H_{\e,1}(t)+H_{\e,2}(t).}
\end{array}\]

According to \eqref{f65} and \eqref{f87}, we have
\[\begin{array}{l}
\ds{\mathbf{E}|Y_{\a,1}(s)|_{H_\gamma}^p\leq c_p\,\mathbf{E}\le(\,\sum_{j=1}^\infty\int_0^s (s-\si)^{-2\a} \le|\le[S_\e(s-\si)-\bar{S}(s-\si)^\vee\r]G(\bar{u}(\si)^\vee)\widehat{u_j \mu}\r|_{H_\gamma}^2\,d\si\r)^{\frac p2}}\\
\vs
\ds{\leq  c_p\,\mathbf{E}\le(\,\sum_{j=1}^\infty\int_0^s (s-\si)^{-2\a} \le|G(\bar{u}(\si)^\vee)\widehat{u_j \mu}\r|_{H_\gamma}^2\,d\si\r)^{\frac p2}.}
\end{array}\]
Since 
\[\le(\,\sum_{j=1}^\infty\int_0^s (s-\si)^{-2\a} \le|G(\bar{u}(\si)^\vee)\widehat{u_j \mu}\r|_{H_\gamma}^2\,d\si\r)^{\frac p2}\leq c_p(T)\le(1+\sup_{\si \in\,[0,s]}|\bar{u}(\si)^\vee|_{H_\gamma}^p\r),\]
due  to  \eqref{f96} and the dominated convergence theorem, we have
\[\lim_{n\to\infty}
\mathbf{E}\le(\,\sum_{j=n+1}^\infty\int_0^s (s-\si)^{-2\a} \le|G(\bar{u}(\si)^\vee)\widehat{u_j \mu}\r|_{H_\gamma}^2\,d\si\r)^{\frac p2}=0.\]
Therefore, for every $\eta>0$ we can fix $n_\eta \in\,\nat$ such that 
\[\begin{array}{l}
\ds{\mathbf{E}|Y_{\a,1}(s)|_{H_\gamma}^p\leq \eta}\\
\vs
\ds{+ c_p(T)\,\mathbf{E}\le(\,\sum_{j=n_\eta+1}^\infty\int_0^s (s-\si)^{-2\a} \le|\le[S_\e(s-\si)-\bar{S}(s-\si)^\vee\r]G(\bar{u}(\si)^\vee)\widehat{u_j \mu}\r|_{H_\gamma}^2\,d\si\r)^{\frac p2}.}
\end{array}\]
Once fixed $n_\eta$, due to \eqref{f70} and the dominated convergence theorem, we have 
\[\lim_{\e\to 0}\mathbf{E}\le(\,\sum_{j=n_\eta+1}^\infty\int_0^s (s-\si)^{-2\a} \le|\le[S_\e(s-\si)-\bar{S}(s-\si)^\vee\r]G(\bar{u}(\si)^\vee)\widehat{u_j \mu}\r|_{H_\gamma}^2\,d\si\r)^{\frac p2}=0,\]
and this, due to the arbitrariness of $\eta$, allows us to conclude that 
\begin{equation}
\label{f110}
\lim_{\e\to 0}\sup_{t \in\,[0,T]} H_{\e,1}(t)=0.
\end{equation}
As far $H_{\e,2}(t)$ is concerned, due to  \eqref{f65} and \eqref{f87}
for every $0<\tau\leq s\leq t$ we have
\[\begin{array}{l}
\ds{\int_0^s \le|\le[S_\e(s-r)-\bar{S}(s-r)^\vee\r] Y_{\a,2}(r)\r|_{H_\gamma}^p\,dr}\\
\vs
\ds{=\int_0^{s-\tau} \le|\le[S_\e(s-r)-\bar{S}(s-r)^\vee\r] Y_{\a,2}(r)\r|_{H_\gamma}^p\,dr+\int_{s-\tau}^s \le|\le[S_\e(s-r)-\bar{S}(s-r)^\vee\r] Y_{\a,2}(r)\r|_{H_\gamma}^p\,dr}\\
\vs
\ds{\leq \int_0^{T}\sup_{\rho \in\,[\tau,T]} \le|\le[S_\e(\rho)-\bar{S}(\rho)^\vee\r] Y_{\a,2}(r)\r|_{H_\gamma}^p\,dr+c_p(T)\sqrt{\tau} \le(\int_0^T \le|Y_{\a,2}(r)\r|_{H_\gamma}^{2p}\,dr\r)^{\frac 12}.}
\end{array}\]
Moreover, if $0\leq s\leq \tau$, we have
\[\int_0^s \le|\le[S_\e(s-r)-\bar{S}(s-r)^\vee\r] Y_{\a,2}(r)\r|_{H_\gamma}^p\,dr\leq c_p(T)\sqrt{\tau} \le(\int_0^T \le|Y_{\a,2}(r)\r|_{H_\gamma}^{2p}\,dr\r)^{\frac 12}.\]
Since, for every $q\geq 1$
\begin{equation}
\label{111}
\mathbf{E}\int_0^T|Y_{\a,2}(r)|^q_{H_\gamma}\,dr<\infty,\end{equation}
this implies
\[H_{\e,2}(t)\leq \mathbf{E} \int_0^{T}\sup_{\rho \in\,[\tau,T]} \le|\le[S_\e(\rho)-\bar{S}(\rho)^\vee\r] Y_{\a,2}(r)\r|_{H_\gamma}^p\,dr+c_p(T)\sqrt{\tau}.\]
Now, let us fix any $\eta>0$ and $\tau_\eta>0$ such that $c_p(T)\sqrt{\tau_\eta}<\eta$.
Because of \eqref{f70}, we have
\[\lim_{\e\to 0} \sup_{\rho \in\,[\tau_\eta,T]} \le|\le[S_\e(\rho)-\bar{S}(\rho)^\vee\r] Y_{\a,2}(r)\r|_{H_\gamma}^p=0.\]
Moreover, as 
\[\sup_{\rho \in\,[\tau_\eta,T]}\le|\le[S_\e(\rho)-\bar{S}(\rho)^\vee\r] Y_{\a,2}(r)\r|_{H_\gamma}^p\leq c_p(T) \,|Y_{\a,2}(r)|^p_{H_\gamma},\]
from \eqref{111} and the dominated convergence theorem, we get
\[\liminf_{\e\to 0}\sup_{ t \in\,[0,T]}H_{\e,2}(t)\leq \lim_{\e\to 0} \mathbf{E} \int_0^{T}\sup_{\rho \in\,[\tau,T]} \le|\le[S_\e(\rho)-\bar{S}(\rho)^\vee\r] Y_{\a,2}(r)\r|_{H_\gamma}^p\,dr+\eta=\eta.\]
Due to the arbitrariness of $\eta>0$, together with \eqref{f110}, this implies that 
\[\lim_{\e\to 0}\mathbf{E}\sup_{s \in\,[0,t]}|J_{\e,2}(s)|_{H_\gamma}^p=0,\]
and this, together with \eqref{f109}, implies \eqref{f933}.
\end{proof}

\end{document}